\documentclass[a4paper,notitlepage,twoside,reqno,12pt]{amsart}
\usepackage{CJK}
\usepackage{mathrsfs}
\usepackage{bbm,pifont,latexsym}
\usepackage{anysize}
\marginsize{2.4cm}{2.4cm}{2.4cm}{2.4cm}
\usepackage{CJK}
\usepackage{dcolumn,indentfirst,color,enumerate,lineno}
\usepackage[pdfpagemode=UseNone,pdfstartview=FitH]{hyperref}
\usepackage{amsmath,amssymb,amscd,amsthm,amsfonts,mathrsfs}
\usepackage{color,graphicx,xcolor,graphics}
\usepackage{titlesec,comment}
\usepackage{titletoc}
\titlecontents{section}[20pt]{\addvspace{2pt}\filright}
              {\contentspush{\thecontentslabel. }}
              {}{\titlerule*[8pt]{}\contentspage}

\usepackage{fancyhdr}
\newtheorem{thm}{Theorem}[section]
\newtheorem{lem}[thm]{Lemma}
\newtheorem{cor}[thm]{Corollary}
\newtheorem{prop}[thm]{Proposition}
\newtheorem{rem}[thm]{Remark}
\newtheorem{rmk}[thm]{Remark}

\theoremstyle{definition}

\newcommand{\BD}[1]{\mathbf{#1}}
\newcommand{\norm}[1]{\left\lVert#1\right\rVert}
\newcommand{\N}{\mathbb{N}}

\newcommand{\R}{\mathbb{R}}
\newcommand{\Z}{\mathbb{Z}}
\newcommand{\Q}{\mathbb{Q}}

\newcommand{\MM}{\mathcal{M}}

\newcommand{\MY}{\mathcal{Y}}

\makeatletter\@addtoreset{equation}{section}\makeatother

\titleformat{\section}{\centering\normalsize}{\textsc{\thesection.}}{0.5em}{\textsc}
\titleformat{\subsection}[runin]{\normalsize}{\textbf{\thesubsection.}}{0.3em}{\textbf}

\begin{document}
\begin{CJK*}{GBK}{song}
\author{Wen-Long Li \\ School of Mathematics\\ Sun Yat-Sen University, Guangzhou, 510275, P. R. China\\ 
Xiaojun Cui\\ Department of Mathematics\\ Nanjing University, Nanjing, 210093, P. R. China}
\email{liwenlongchn@gmail.com}

\email{xcui@nju.edu.cn}

\title{Multitransition solutions for a generalized Frenkel-Kontorova model}

\begin{abstract}
We study a generalized Frenkel-Kontorova model.
Using minimal and Birkhoff solutions as building blocks,
we construct a lot of homoclinic solutions and heteroclinic solutions for this generalized Frenkel-Kontorova model under gap conditions.
These new solutions are not minimal and Birkhoff any more.
We use constrained minimization method to prove our results.
\end{abstract}



\date{September 10, 2019}

\maketitle


\section{Introduction}\label{sec:intro}

In recent years, a generalized $n$-dimensional (or $n$-D for short, with $n\geq 2$) Frenkel-Kontarova (or FK for short) model has been extensively studied (cf. e.g., \cite{dela2, MR2012JDE, Miao, LC}).
In particular, part of the results of Moser-Bangert theory has been established for this model (cf. \cite{MR2012JDE, Miao, LC}, see also \cite{dela1}).
By Moser-Bangert theory, following Rabinowitz and Stredulinsky (\cite{RS}),
we mean an elegant theory initialed by Moser (\cite{Moser}), and extended by Bangert (\cite{bangert1987, bangertcmh, Bangert}),
and generalized by Rabinowitz and Stredulinsky (\cite{RS3, RS4, RS5, RS6, RS2003, RS2004, RS}) and other researchers.
Moser-Bangert theory also has many applications, cf. \cite{caff, caff2001, val, tor, bes, llavevaldi, cozzi} and the references therein.
Moser and Bangert considered a variational problem
and studied the minimal and without self-intersections solutions of this problem.
They clarified the structure of the set of these solutions.
Rabinowitz and Stredulinsky studied an Allen-Cahn type equation, which belongs to the classes of the variational problem of Moser and Bangert.
Rabinowitz and Stredulinsky obtained a lot of homoclinic and heteroclinic solutions of the Allen-Cahn type equation.
Note that although Rabinowitz and Stredulinsky studied a special class of the variational problem of Moser and Bangert,
their methods and results hold for more generalized equations besides their Allen-Cahn type equation.

In this paper, we use variational method
to obtain more homoclinic and heteroclinic solutions of a generalized $n$-D FK model.
The construction of these new solutions are based on minimal and Birkhoff solutions of \cite{LC}.
The method of this paper follows \cite{RS}.
This method is also used in dynamical systems by Mather (\cite{Mather1993}) for constructing heteroclinic orbits.
We recall some definitions and results of the generalized FK model considered in \cite{LC}.

\subsection{Generalized FK model.}
\

The generalized FK model considered in this paper is a problem on the lattice $\Z^n$.
To introduce this problem, we recall some notations.
A configuration is a function $u: \Z^n \to \R$ and we denote such a function by $u\in\R^{\Z^n}$. Similarly we can define $u\in\R^{E}$ for any $E\subset \Z^n$.
We use $\BD{i},\BD{j},\BD{k}$, etc. (resp. $i,j,k$, etc.) to represent elements in $\Z^n$ (resp. $\Z$).
For $\BD{i}\in\Z^n$, we set $\norm{\BD{i}}:=\sum_{k=1}^{n}|\BD{i}_k|$.
Fix $r\in \N$ and let $B_{\BD{0}}^{r}=\{\BD{k}\in\Z^{n}\,|\, \norm{\BD{k}}\leq r\}.$
We introduce a local potential as follows (cf. \cite{LC, MR2012JDE, Miao}).
Assume that $s\in C^2(\R^{B_{\BD{0}}^{r}}, \R)$ satisfies:
\begin{enumerate}[({S}1)]
  \item \label{eq:S1} $s(u+1_{B_{\BD{0}}^{r}})=s(u)$, where $1_{B_{\BD{0}}^{r}}$ is the constant function $1$ on $B_{\BD{0}}^{r}$;
  \item \label{eq:S2} $s$ is bounded from below and coercive in the following sense,
      \begin{equation*}
        \lim_{|u(\BD{k})-u(\BD{j})|\to \infty}s(u)=\infty, \textrm{ for $\BD{k}, \BD{j}\in B_{\BD{0}}^{r}$ with $\norm{\BD{k}-\BD{j}}=1$;}
      \end{equation*}

  \item \label{eq:S3} $\partial_{\BD{k}, \BD{j}}s \leq 0$ for $\BD{k}, \BD{j}\in B_{\BD{0}}^{r}$ with $\BD{k}\neq \BD{j}$, while $\partial_{\BD{0},\BD{j}}s<0$ for any $\BD{j}$ with $\norm{\BD{j}}=1$.
\end{enumerate}
For $u\in\R^{\Z^n}$, set $S_{\BD{j}}(u)=s(\tau_{-\BD{j}_n}^{n}\cdots \tau_{-\BD{j}_1}^{1}u|_{B_{\BD{0}}^{r}})$, where $\tau_{-k}^{j}: \R^{\Z^n}\to \R^{\Z^n}$ is defined by $\tau_{-k}^{j}u(\BD{i})=u(\BD{i}+k\BD{e}_j)$. With these locally potentials $S_{\BD{j}}$, we define a formal sum
\begin{equation}\label{eq:potential}
\sum_{\BD{j}\in\Z^n}S_{\BD{j}}(u).
\end{equation}
The
Euler-Lagrange equation of \eqref{eq:potential} is
\begin{equation}\label{eq:PDE}
    \sum_{\BD{j}\in\Z^n}\partial_{\BD{i}}S_{\BD{j}}(u)=\sum_{\BD{j}:\norm{\BD{j}-\BD{i}}\leq r}\partial_{\BD{i}}S_{\BD{j}}(u)=0
\end{equation}
for all $\BD{i}\in\Z^n$.
Note that \eqref{eq:PDE} always makes sense since the sum in this equation only involves finite terms.
An example of \eqref{eq:PDE} is
\begin{equation}\label{eq:exam}
-\frac{1}{2n}\sum_{\BD{j}: ||\textbf{j}-\textbf{i}||=1}(u(\textbf{j})-u(\textbf{i}))+ V'(u(\textbf{i}))=0,
\end{equation}
where $V\in C^{2}(\mathbb{R},\mathbb{R})$ is $1$-periodic.
Setting
\begin{equation*}
s(u|_{B_{\BD{0}}^{1}})=V(u(\BD{0}))+\frac{1}{8n}\sum_{\BD{k}: \norm{\BD{k}}=1}(u(\BD{k})-u(\BD{0}))^2
\end{equation*}
and letting $S_j$ as above, we have that \eqref{eq:exam} is exact the Euler-Lagrange equation of \eqref{eq:potential}.
\eqref{eq:exam} is an $n$-D form of the classical $1$-D FK model.

\subsection{Minimal and Birkhoff solutions in \cite{LC}.}
\

In \cite{LC}, we used the minimization methods developed by Rabinowitz and Stredulinsky (\cite{RS}) to construct heteroclinic solutions for a generalized FK model.
In this paper, we will construct some multitransition solutions, a term by Rabinowitz and Stredulinsky, by the methods in \cite{RS}.

The first theme of Moser-Bangert theory is to study minimal and without self-intersections (Birkhoff, in our case) solutions.
For $v\in \R^{\Z^n}$, define $supp(v)=\{\BD{i}\,|\, v(\BD{i})\neq 0\}$.
We say that $v$ has compact support if $supp(v)$ is a bounded set of $\Z^n$.
A point $\BD{i}$ of $E(\subset \Z^n)$ is called to be an interior point of $E$ if
\begin{equation*}
  \BD{i}+B_{\BD{0}}^{r}:=\{\BD{i}+\BD{j}\,|\, \BD{j}\in B_{\BD{0}}^{r} \}\subset E.
\end{equation*}
We denote all the interior points of $E$ by $int(E)$.
A configuration $u\in\R^{\Z^n}$ is said to be minimal, if for any $v\in\R^{\Z^n}$ with compact support,
\begin{equation*}
  \sum_{\BD{j}\in E}S_{\BD{j}}(u)\leq \sum_{\BD{j}\in E}S_{\BD{j}}(u+v)
\end{equation*}
holds
for any bounded set $E$ with $supp(v)\subset int(E)$.

To introduce Birkhoff configuration, we define some partial ordered relations in $\R^{\Z^n}$ as follows.
We say $u\leq (\textrm{ or } <, \textrm{ or } =) v$ if $u(\BD{j})\leq (\textrm{ or } <, \textrm{ or } =) v(\BD{j})$ for all $\BD{j}\in\Z^n$;
$u\lneqq v$ if $u\leq v$ and there exists some $\BD{k}$ such that $u(\BD{k})\neq v(\BD{k})$;
Similarly one define $\geq, \gneqq, >$ in $\R^{\Z^n}$.
We say $u$ is Birkhoff if for any $k\in\Z$ and for any $j=1,2,\cdots,n$, one and only one of the following holds:
 \begin{equation*}
   \tau_{-k}^{j}u<u\quad\quad \textrm{or} \quad \quad  \tau_{-k}^{j}u=u\quad\quad \textrm{or} \quad \quad \tau_{-k}^{j}u>u.
 \end{equation*}

The unit vector with $j$th component $1$ and other components $0$ is denoted by $\BD{e}_j$.
If $u$ is $1$-periodic in all directions, that is, $u(\BD{i}+\BD{e}_j)=u(\BD{i})$ for all $\BD{i}\in \Z^n$ and $j=1,2,\cdots, n$,
then it will be denoted by $u\in \R^{(\Z/\{1\})^{n}}$.
Similarly we can define $\R^{(\Z/\{k\})^{n}}$, $\R^{\Z\times (\Z/\{1\})^{n-1}}$, etc.
The most important feature of Birkhoff configuration is that it has a rotation vector.
Rotation vector is an analogue of rotation number of Aubry-Mather theory
and is used in the clarification of minimal and Birkhoff solutions.
For $u\in\R^{\Z^n}$, if the limit
\[
\lim_{|m|\to\infty}\frac{u(m\BD{i})}{m}
\]
exists and equals $\langle \alpha,\BD{i}\rangle$,
we say $u$ has rotation vector $\alpha$.
To state and prove our main results simply,
we take $\alpha=\BD{0}$
and will indicate how to generalize $\alpha=\BD{0}$ to $\alpha\in\Q$ at the end of this paper.
In recent years, minimal and Birkhoff solutions of \eqref{eq:PDE} are carefully studied, cf. \cite{dela1, dela2, MR2012JDE, Miao, LC}.
We mention some results related to this paper.
In \cite{LC} (see also \cite{MR2012JDE}), the authors obtained periodic solutions with rotation vector $\alpha=\BD{0}$, denoted by $\MM_{\BD{0}}$.
$\MM_{\BD{0}}$ is an ordered set. That is, for any $u,v\in\MM_{\BD{0}}$, $u<v$, or $u=v$, or $u>v$.
To construct heteroclinic solutions, we need a gap condition:
\begin{equation}\label{eq:*0}
\textrm{there are adjacent $v_0$, $w_0\in \mathcal{M}_0$ with $v_0 < w_0$}.\tag{$*_0$}
\end{equation}
Recall that in an ordered set $A$, $v,w\in A$ with $v<w$ are
said to be adjacent if there is no element of $A$ lying between $v$ and $w$.

In \cite{LC}, the authors proved that if \eqref{eq:*0} holds, there are heteroclinic solutions lying
between $v_0,w_0$ and asymptotic to $v_0$ (resp. $w_0$) as $\BD{i}_1\to -\infty$ and to $w_0$ (resp. $v_0$) as $\BD{i}_1\to \infty$,
and these solutions are denoted by $\MM_1(v_0,w_0)$ (resp. $\MM_1(w_0,v_0)$).
Note the elements in $\MM_1(v_0,w_0)$ and $\MM_1(w_0,v_0)$ are $1$-periodic in $\BD{i}_2, \cdots, \BD{i}_n$.
It is also proved in \cite{LC} that $\MM_1(w_0,v_0)$ and $\MM_1(w_0,v_0)$ are ordered.
To construct more complex heteroclinic solutions, we need
\begin{equation}\label{eq:*1}
\begin{split}
    \textrm{there are adjacent } &v_1, w_1 \in \MM_1(v_0,w_0) \textrm{ with } v_1<w_1, \\
    \textrm{ and there are adjacent } &\tilde{v}_1, \tilde{w}_1 \in \MM_1(w_0,v_0) \textrm{ with } \tilde{v}_1<\tilde{w}_1 .
    \end{split}\tag{$*_1$}
\end{equation}
If $s$ does not satisfy \eqref{eq:*0} and \eqref{eq:*1}, one can perturb $s$ to obtain these conditions.
If $s$ satisfies \eqref{eq:*0} and \eqref{eq:*1}, then for all $\bar{s}$ close to $s$ in some sense, \eqref{eq:*0} and \eqref{eq:*1} are satisfied by $\bar{s}$.
Please see \cite{LC} for more discussions.

\subsection{Main results.}\label{sec:00003}
\

Now our main result of this paper can be stated. \\

\emph{Suppose $s\in C^2(\R^{B^{r}_{\BD{0}}},\R)$ satisfies (S\ref{eq:S1})-(S\ref{eq:S3}),
and \eqref{eq:*0} \eqref{eq:*1} holds.
Then
\begin{enumerate}
  \item there are infinitely many homoclinic solutions asymptotic to $v_0$
(resp. to $w_0$) as $|\BD{i}_1| \to \infty$ and $1$-periodic in $\BD{i}_2,\cdots,\BD{i}_n$; \label{item:1}
  \item there are infinitely many solutions of \eqref{eq:PDE} that asymptotic to $v_0$ (resp. to $w_0$) as $\BD{i}_1 \to -\infty$
and to $w_0$ (resp. to $v_0$) as $\BD{i}_1 \to \infty$, and $1$-periodic in $\BD{i}_2,\cdots,\BD{i}_n$.  \label{item:2}
\end{enumerate}}

The basic heteroclinic solutions in $\MM_{1}(v_0,w_0)$ and $\MM_{1}(w_0,v_0)$ are $1$ transition solutions.
The homoclinic solutions obtained in \eqref{item:1} 
are $2k$ transition ($k\geq 1$).
For solutions homoclinic to $v_0$, $2k$ transition means it will experience $2k$ times phase transitions
before returning back to $v_0$.
Similarly, the heteroclinic solutions of \eqref{item:2} are $(2k+1)$ transition ($k\geq 1$).

Note that if $v$ is a solution of \eqref{eq:PDE}, so is $\tau_{k}^{1}v$ for any $k\in\Z$.
We say $\tau_{k}^{1}v$ and $v$ are not geometrically distinct (\cite{Ra}).
But in our results,
there are infinitely many geometrically distinct solutions.
Please see Remark \ref{rem:6.10} below.

This paper is organized as follows.
Section \ref{sec:pre} gives some preliminaries needed for proving the existence of multitransition solutions of \eqref{eq:PDE}
and section \ref{chap:7} is devoted to prove the existence of $2$ transition solutions.
In section \ref{chap:8}, we illustrate
the existence of general $k$ transition solutions of \eqref{eq:PDE} and give some generalizations.

\section{Preliminaries}\label{sec:pre}

We prove our main theorem by constrained minimization method
that will be stated
in this section.
Before that, we recall some facts about the generalized FK model.

\begin{lem}[{cf. \cite[Lemma 2.6]{Miao}, \cite[Lemma 2.8]{LC}}]\label{lem:2.6miao}
For $u, v\in\R^{\Z^n}$ and for any finite set $B\subset \Z^n$, we have
\begin{equation*}
   \sum_{\BD{j}\in B}S_{\BD{j}}(\phi)+\sum_{\BD{j}\in B}S_{\BD{j}}(\psi)\leq \sum_{\BD{j}\in B}S_{\BD{j}}(u)+\sum_{\BD{j}\in B}S_{\BD{j}}(v),
\end{equation*}
where $\phi, \psi$ are defined by $\phi=\max(u,v)$, $\psi=\min(u,v)$.
\end{lem}
For $v,w\in \MM_0$ with $v<w$, set
\[
\hat{\Gamma}_1(v,w)=\{u\in \R^{\Z\times (\Z/\{1\})^{n-1}}\,|\,v\leq u\leq w\},
\]
and
\[
\begin{split}
\Gamma_1(v,w)=\{u\in \hat{\Gamma}_1(v,w)\,|\, &\norm{u-v}_{\BD{T}_i}\to 0, \quad i\to -\infty, \\
& \norm{u-w_0}_{\BD{T}_i}\to 0, \quad i\to \infty\}.
\end{split}
\]
Here $\BD{T}_i=i\BD{e}_1$ and $\norm{u}_{\BD{j}}=|u(\BD{j})|$.
Define \[
c_0=\inf_{u\in \R^{(\Z/\{1\}) ^{n}}}S_{\BD{0}}(u).
\]
For $u\in \hat{\Gamma}_1(v,w)$, $p,q\in \Z$ with $p\leq q$, define
\begin{equation*}
    J_{1,p}(u)= S_{\BD{T}_p}(u)-c_0, \quad
      J_{1;p,q}(u)= \sum_{j=p}^{q}J_{1,j}(u),
\end{equation*}
and
\begin{equation}\label{eq:defofJ1}
  J_1(u) = \liminf_{p\to -\infty \atop q\to \infty}J_{1;p,q}(u).
\end{equation}
The next lemma shows that $J_1$ is well-defined for $u\in\hat{\Gamma}_1(v,w)$.
\begin{lem}[{cf. \cite[Propositions 3.2, 3.4 and Lemma 3.3]{LC}}]\label{lem:23}
\begin{enumerate}
  \item If $u\in\hat{\Gamma}_1(v,w)$ and $p\leq q\in\Z$, there is a constant $K_1=K_1(v,w) \geq 0$, such that
\begin{equation*}
    -K_1\leq J_{1;p,q}(u)\leq J_1(u)+2K_1.
\end{equation*}
  \item If $u\in \Gamma_1(v,w)$, $J_{1,i}(u)\to 0$ as $|i|\to \infty$.
If $u\in \Gamma_1(v,w)$ and $J_1(u)<\infty$,
\[
J_1(u)=\lim_{p\to -\infty \atop q\to \infty}J_{1;p,q}(u),
\]
 that is, the $\liminf$ in \eqref{eq:defofJ1} becomes limit.
\end{enumerate}
\end{lem}
When one wants to apply minimization method,
one of the difficulties is to show that a minimization sequence has a convergent subsequence.
But in our case, it is easy to overcome this difficulty.
\begin{lem}[{cf. \cite[Proposition 3.7]{LC}}]\label{lem:2.50}
Let $\MY\subset \hat{\Gamma}_1 (v,w)$ and define
\begin{equation}\label{eq:2.51}
c(\MY)=\inf_{u\in \MY}J_1 (u).
\end{equation}
Suppose $(u_k)$ is a minimizing sequence for \eqref{eq:2.51},
then there is a $U\in \hat{\Gamma}_1(v,w)$ such that along a subsequence, $u_k \to U$ pointwise.
If $c(\MY)<\infty$, then
\begin{equation*}
  -K_1\leq J_1(U)\leq c(\MY)+1+2K_1,
\end{equation*}
where $K_1$ is defined in Lemma \ref{lem:23}.
\end{lem}
The next proposition tells us how to verify
a minimizer of a suitable functional over a
set is a solution of \eqref{eq:PDE}.
For $i\in\Z$, define
\begin{equation}\label{eq:delta}
  \delta_{\BD{T}_i}(\BD{j})=\left\{
                      \begin{array}{ll}
                        1, & \BD{j}_1=i, \\
                        0, & \BD{j}_1\neq i.
                      \end{array}
                    \right.
\end{equation}
\begin{lem}[{cf. \cite[Proposition 3.8]{LC}}]\label{lem:2.64}
Let $\MY\subset \hat{\Gamma}_1 (v,w)$.
If $c(\MY)<\infty$ and there is a minimizing sequence $(u_k)$ for $c(\MY)$ such that for some $i\in \Z $,
the function $\delta_{\BD{T}_i}$ and some $t_0>0$, we have
\begin{equation}\label{eq:2.65}
c(\MY)\leq J_{1} (u_k +t\delta_{\BD{T}_i})+\epsilon_k
\end{equation}
for all $|t|\leq t_0$, where $\epsilon_k \to 0$ as $k\to \infty$.
Then the limit $U$ of $u_k$ satisfies \eqref{eq:PDE} at $\BD{T}_i$.
Moreover, $U$ satisfies \eqref{eq:PDE} at any $\BD{j}$ with $\BD{j}_1=(\BD{T}_i)_1=i$.
\end{lem}

We have the following strong comparison result, which is very important in our analysis.
\begin{lem}[{cf. \cite[Lemma 2.5]{Miao}, \cite[Lemma 4.5]{MR2012JDE}, \cite[Lemma 2.6]{LC}}]\label{lem:unknown}
Assume that $u,v$ are solutions of \eqref{eq:PDE} and $u\leq v$.
Then either $u<v$ or $u= v$.
\end{lem}

\begin{cor}[{cf. \cite[Corollary 2.7]{LC}}]\label{cor:bijiao}
Assume that $u, v$ are solutions of \eqref{eq:PDE}.
If $\psi:=\min(u,v)$ or $\phi:=\max(u,v)$ is a solution of \eqref{eq:PDE}, then
\begin{equation*}
    u<v, \textrm{\quad or \quad} u=v, \textrm{\quad or \quad} u>v.
\end{equation*}
\end{cor}

To introduce a useful comparison result that appears repeatedly, for $v\in\MM_0$, we define
\begin{equation*}
  \Gamma_1(v)=\{u\in\hat{\Gamma}_1(v-1,v+1)\,|\, \norm{u-v}_{\BD{T}_i}\to 0, \textrm{ as } |i|\to \infty\},
\end{equation*}
and
\begin{equation*}
  c_1(v)=\inf_{u\in\Gamma_1(v)}J_1(u),\quad \textrm{and} \quad \MM_1(v)=\{u\in\Gamma_1(v)\,|\, J_1(u)=c_1(v)\}.
\end{equation*}

\begin{lem}[{cf. \cite[Theorem 3.11]{LC}}]\label{lem:2.72}
If $s\in C^2(\R^{B^{r}_{\BD{0}}},\R)$ satisfies (S\ref{eq:S1})-(S\ref{eq:S3}), then $c_1(v)=0$ and $\MM_1(v)=\{v\}$.
\end{lem}

We are now in a position to state the constrained variational problem.
The case of $2$ transition solution will be treated in detail
and $k$ ($k>2$) transition solution will be sketched in Section \ref{chap:8}.
Recall that under the gap conditions \eqref{eq:*0}-\eqref{eq:*1}, $\MM_1(v_0,w_0)$ and $\MM_1(w_0,v_0)$ are ordered sets.
If we set
\begin{equation*}
  \rho_{-}(u)=\norm{u-v_0}_{\BD{T}_0} \quad\textrm{ and }\quad \rho_{+}(u)=\norm{u-w_0}_{\BD{T}_0},
\end{equation*}
then $\rho_{-}, \rho_{+}$ are monotone on $\MM_1(v_0,w_0)$ and $\MM_1(w_0,v_0)$.
Let $\bar{\rho}=\norm{w_0-v_0}_{\BD{T}_0}$.
By \eqref{eq:*1}, $\MM_1(v_0,w_0)$ and $\MM_1(w_0,v_0)$ do not contain a continuum of members,
so we can take
$\rho_i \in (0,\bar{\rho})$, $1\leq i \leq 4$, satisfying
\begin{equation}\label{eq:6.2}
\begin{split}
  \rho_1\not\in \rho_{-}(\MM_1(v_0,w_0)), & \quad\quad\rho_2\not\in\rho_{+}(\MM_1(v_0,w_0)), \\
   \rho_3\not\in \rho_{+}(\MM_1(w_0,v_0)), & \quad\quad\rho_4\not\in\rho_{-}(\MM_1(w_0,v_0)).
\end{split}
\end{equation}
For $l\in\N$, let $\BD{m}=(\BD{m}_1,\BD{m}_2,\BD{m}_3,\BD{m}_4)\in \Z^4$ satisfy
\begin{equation}\label{eq:6.3}
\BD{m}_1<\BD{m}_2<\BD{m}_2+2l<\BD{m}_3<\BD{m}_4.
\end{equation}

The functional space of the constrained variational minimization problem is defined as follows.
Let
\begin{equation}\label{eq:6.4}
Y_{\BD{m},l}:= Y_{\BD{m},l}(v_0,w_0):= \{u\in \hat{\Gamma}_1(v_0,w_0)\,|\, u \textrm{ satisfies } \eqref{eq:6.5}-\eqref{eq:6.6}\},
\end{equation}
where
\begin{equation}\label{eq:6.5}
  \left\{
    \begin{array}{ll}
(a) \quad\rho_{-}(\tau_{-i}^{1}u)&\leq \rho_1,\quad \BD{m}_1-l\leq i\leq \BD{m}_1-1, \\
(b) \quad\rho_{+}(\tau_{-i}^{1}u)&\leq \rho_2, \quad \BD{m}_2\leq i\leq \BD{m}_2+l-1, \\
(c) \quad\rho_{+}(\tau_{-i}^{1}u)&\leq \rho_3,\quad \BD{m}_3-l\leq i\leq \BD{m}_3-1, \\
(d) \quad\rho_{-}(\tau_{-i}^{1}u)&\leq \rho_4, \quad \BD{m}_4\leq i\leq \BD{m}_4+l-1,
    \end{array}
  \right.
\end{equation}
and
\begin{equation}\label{eq:6.6}
  \norm{u-v_0}_{\BD{T}_i}\to 0,\quad |i|\to \infty. 
\end{equation}
Set
\begin{equation}\label{eq:6.7}
  b_{\BD{m},l}:= b_{\BD{m},l}(v_0,w_0):=\inf_{u\in Y_{\BD{m},l}}J_{1}(u).
\end{equation}

We restate the main result in Section \ref{sec:00003} as some theorems. The first and the simplest theorem is:
\begin{thm}\label{thm:6.8}
Suppose $s\in C^2(\R^{B^{r}_{\BD{0}}},\R)$ satisfies (S\ref{eq:S1})-(S\ref{eq:S3}).
Assume that \eqref{eq:*0} and \eqref{eq:*1} hold.
Then for each sufficiently large $l\in\N$, there is a $U=U_{\BD{m},l}\in Y_{\BD{m},l}$ such that $J_1(U)=b_{\BD{m},l}$.
Moreover, $U$ is a solution of \eqref{eq:PDE} provided that $\BD{m}_2-\BD{m}_1$ and $\BD{m}_4-\BD{m}_3$ are large enough.
\end{thm}
\begin{rem}\label{rem:6.10}
Enlarging $l, \BD{m}_2-\BD{m}_1$ and $\BD{m}_4-\BD{m}_3$ gives infinitely many $2$ transition solutions of \eqref{eq:PDE}.
Of course, these solutions are geometrically distinct.
\end{rem}

We postpone the proof of Theorem \ref{thm:6.8} until section \ref{chap:7}.
The other theorems for proving the main result in Section \ref{sec:00003} will be stated in Section \ref{chap:8}.
The remainder of this section is devoted to give some preliminaries.
Set
\[
c_1(v_0,w_0)=\inf_{u\in\Gamma_1(v_0,w_0)}J_1(u), \quad \textrm{and}\quad c_1(w_0,v_0)=\inf_{u\in\Gamma_1(w_0,v_0)}J_1(u).
\]

\begin{lem}\label{lem:6.11}
Suppose \eqref{eq:*0} holds. Then we have
$c_1(v_0,w_0)+c_1(w_0,v_0)>0$.
\end{lem}
\begin{proof}
By \eqref{eq:*0} and \cite[Theorem 3.13, Remark 3.14]{LC}, $\MM_1(v_0,w_0)$ and $\MM_1(w_0,v_0)$ are not empty.
Then the proof of Lemma \ref{lem:6.11} is the same as \cite[Lemma 6.11]{RS} except for \[J_1(\Phi)+J_1(\Psi)\leq J_1(V)+J_1(W),\]
which follows from Lemmas \ref{lem:2.6miao} and \ref{lem:23}.
\end{proof}

For $k\in\Z$, set \[ X_k:= \cup_{i=-r}^{r}\BD{T}_{k+i}. \]
The following proposition is very useful in comparison arguments.
\begin{prop}\label{prop:6.13}
Suppose \eqref{eq:*0} holds. For any $\gamma\in (0,\bar{\rho})$,
there is a $\beta=\beta(\gamma)>0$ 
such that $J_1(u)\geq \beta$ for any
\[
\begin{split}
&u\in \{u\in\Gamma_1(v_0)\cap\hat{\Gamma}_1(v_0,w_0)\,|\, \norm{u-v_0}_{X_0}\geq \gamma\}, \\
\textrm{or  }&u\in \{u\in\Gamma_1(w_0)\cap\hat{\Gamma}_1(v_0,w_0)\,|\, \norm{u-w_0}_{X_0}\geq \gamma\}.
\end{split}
\]
\end{prop}
\begin{proof}
We only prove the case of $u\in\Gamma_1(v_0)\cap \hat{\Gamma}_1(v_0,w_0)$ since the other case can be proved similarly.
Define
\begin{equation*}
  \MY:=\{u\in\Gamma_1(v_0)\cap \hat{\Gamma}_1(v_0,w_0)\,|\, \norm{u-v_0}_{X_0}\geq \gamma\}
\end{equation*}
and set $c(\MY):=\inf_{u\in\MY}J_1(u)$.
Then by Lemma \ref{lem:2.72},
 $ 0=c_1(v_0)\leq c(\MY)<\infty$,
where
\begin{equation}\label{eq:2.71}
  c_1(v_0):=\inf_{u\in \Gamma_1(v_0)}J_1(u).
\end{equation}
If $c(\MY)>0$,
set $\beta(\gamma):=c(\MY)$ and we are done.
Now suppose, by contradiction, $c(\MY)=0$.
Take $(u_k)\subset \MY$ such that $J_1(u_k)\to c(\MY)=c_1(v_0)=0$ as $k\to \infty$.
Since $\MY\subset \Gamma_1(v_0)$, Lemma \ref{lem:2.50} ensures
there are a subsequence, still denoted by $u_k$, and a
$P\in \hat{\Gamma}_1(v_0,w_0)$ such that $J_1(P)<\infty$ and
\begin{equation}\label{eq:uk}
  u_k\to P \quad \textrm{ pointwise as} \quad k\to \infty.
\end{equation}
So $\norm{P-v_0}_{X_0}\geq \gamma$.

We claim that:
\begin{equation}\label{eq:claim}
  \textrm{$P$ is a solution of \eqref{eq:PDE}}.
\end{equation}
We need to verify that the condition of Lemma \ref{lem:2.64} is satisfied.
A comparison argument as in the proof of (A) of \cite[Theorem 3.13]{LC} will be employed.
Indeed, let $\delta_{\BD{T}_i}$ be as in Lemma \ref{lem:2.64} and
\begin{equation*}
  2|t|\leq \left\{
            \begin{array}{ll}
              \min(1,v_0-w_0+1), & \textrm{ if } v_0>w_0-1 ,\\
              1, & \textrm{ if } v_0=w_0-1.
            \end{array}
          \right.
\end{equation*}
Set $\chi_k=\max(v_0, \min(u_k+t\delta_{\BD{T}_i},w_0))$.
Then $\chi_k,\min(v_0, \min(u_k+t\delta_{\BD{T}_i},w_0))\in\Gamma_1(v_0)$.
Thus
\begin{equation*}
  \begin{split}
   J_1(\chi_k) \leq &J_1(\chi_k)+J_1(\min(v_0, \min(u_k+t\delta_{\BD{T}_i},w_0)))\\
   \leq &J_1(\min(u_k+t\delta_{\BD{T}_i},w_0))\\
   \leq &J_1(\max(u_k+t\delta_{\BD{T}_i},w_0))+J_1(\min(u_k+t\delta_{\BD{T}_i},w_0))\\
  \leq &J_1(u_k+t\delta_{\BD{T}_i}),
 \end{split}
\end{equation*}
where the first and the third inequalities follow from Lemma \ref{lem:2.72}, while the second and the last inequalities follow from Lemma \ref{lem:2.6miao}.
Hence we have
\begin{equation*}
\begin{split}
 & c_1(v_0)\leq  J_1(u_k)=:c_1(v_0)+\epsilon_k\\
\leq & J_1(\chi_k)+\epsilon_k \leq  J_1(u_k+t\delta_{\BD{T}_i})+\epsilon_k,
\end{split}
\end{equation*}
where $\epsilon_k\to 0$ as $k\to \infty$.
Applying Lemma \ref{lem:2.64} shows that \eqref{eq:claim} holds.

It is easy to see that $\max(u_k,\tau_{-1}^{1}u_k), \min(u_k,\tau_{-1}^{1}u_k)\in \Gamma_1(v_0)\cap \hat{\Gamma}_1(v_0,w_0)$.
By Lemma \ref{lem:2.6miao},
\begin{equation*}
\begin{split}
  & J_1(\max(u_k,\tau_{-1}^{1}u_k))+J_1(\min(u_k,\tau_{-1}^{1}u_k))\\
\leq & J_1(u_k)+J_1(\tau_{-1}^{1}u_k)\\
= & 2J_1(u_k)\\
\to& 2c_1(v_0)=0
\end{split}
\end{equation*}
as $k\to \infty$.
Therefore $\max(u_k,\tau_{-1}^{1}u_k)$ and $\min(u_k,\tau_{-1}^{1}u_k)$ are also minimizing sequences for \eqref{eq:2.71}.
Noting that $\max(u_k,\tau_{-1}^{1}u_k)\to \max(P,\tau_{-1}^{1}P)$, $\min(u_k,\tau_{-1}^{1}u_k)\to \min(P,\tau_{-1}^{1}P)$ pointwise as $k\to \infty$,
by the arguments proving \eqref{eq:claim}, $\max(P,\tau_{-1}^{1}P)$ and $\min(P,\tau_{-1}^{1}P)$ are solutions of \eqref{eq:PDE}.
By Corollary \ref{cor:bijiao}, we have
\begin{equation*}
  (a)\,\, P=\tau_{-1}^{1}P \quad \textrm{  or  }\quad  (b)\,\, P<\tau_{-1}^{1}P \quad \textrm{  or  }\quad  (c)\,\, P>\tau_{-1}^{1}P.
\end{equation*}
If (a) is satisfied, $P\in\Gamma_0$.
Since $J_1(P)<\infty$, $J_1(P)=0$, thus $P=v_0$ or $P=w_0$.
Note $\norm{P-v_0}_{X_0}\geq \gamma$, so $P=w_0$.
Case (b) and case (c) are proved similarly, so we only prove case (b).
If (b) holds, then $P\in \hat{\Gamma}_1(v_0,w_0)\setminus \{v_0,w_0\}$.
Noting $J_1(P)<\infty$, \cite[Corollary 3.6]{LC} implies $P\in\Gamma_1(v_0,w_0)$.
Thus
\begin{equation}\label{eq:6.17}
  \norm{P-w_0}_{\BD{T}_i}\to 0, \quad \quad\textrm{as } i\to \infty.
\end{equation}
Since \eqref{eq:6.17} also holds for $P=w_0$ of case (a), to complete the proof of Proposition \ref{prop:6.13}, we shall show \eqref{eq:6.17} leads to a contradiction.

For any $\epsilon >0$, by \eqref{eq:uk} and \eqref{eq:6.17}, there is a $q=q(\epsilon)\in\N$ such that for $k\in\N$ large enough,
\begin{equation}\label{eq:6.18}
  \sum^{q+2+r}_{i=q-1-r}\norm{u_k-w_0}_{\BD{T}_i}\leq \epsilon.
\end{equation}
Set
\begin{equation*}
  g_k=\left\{
        \begin{array}{ll}
          u_k, & \BD{i}_1\leq q, \\
          w_0, & q<\BD{i}_1,
        \end{array}
      \right.
\end{equation*}
and
\begin{equation*}
  h_k=\left\{
        \begin{array}{ll}
          w_0, & \BD{i}_1\leq q, \\
          u_k, & q<\BD{i}_1.
        \end{array}
      \right.
\end{equation*}
We have $g_k\in\Gamma_1(v_0,w_0)$, $h_k\in\Gamma_1(w_0,v_0)$.
Then by \eqref{eq:6.18}, for $k$ large enough, there is a function $\kappa(\theta)$ satisfying $\kappa(\theta)\to 0$ as $\theta\to 0$ such that
\begin{equation*}
\begin{split}
0<& c_1(v_0,w_0)+c_1(w_0,v_0)\\
\leq& J_1(g_k)+J_1(h_k)\\
\leq& J_1(u_k)+\kappa(\epsilon).
\end{split}
\end{equation*}
Letting $k\to \infty$ and then $\epsilon\to 0$ leads to a contradiction.
Thus $c(\MY)>0$ and we complete the proof of Proposition \ref{prop:6.13}.
\end{proof}

As in \cite{RS}, the following proposition is important for proving Theorem \ref{thm:6.8}.

\begin{prop}\label{prop:6.27}
Suppose \eqref{eq:*0} holds and $u\in\hat{\Gamma}_1(v_0,w_0)$ with $J_1(u)\leq M<\infty$.
Then for any $\sigma>0$ and $t\in\Z$, there is an $l_0=l_0(\sigma, M)\in \N$ independent of $u$ and $t$ such that whenever $l\in\N$ and $l\geq l_0$,
\begin{equation*}
  \norm{u-\phi}_{X_i}< \sigma
\end{equation*}
for some $i=i(l,t)\in (t-l/2,t+l/2)$ and $\phi=\phi_{l,t}\in \{v_0,w_0\}$.
\end{prop}

\begin{proof}
Suppose, by contradiction, there exist a $\sigma>0$, $t\in\Z$, and a sequence $(u_k)\subset\hat{\Gamma}_1(v_0,w_0)$ such that
\begin{equation*}
  J_1(u_k)\leq M
\end{equation*}
and
\begin{equation*}
  \norm{u_k-\phi}_{X_i}\geq \sigma
\end{equation*}
for $\phi\in \{v_0, w_0\}$ and for any $i\in(t-k,t+k)$.
Note that $\phi$ does not depend on $k$, if not, replace $(u_k)$ by a subsequence.
Since $(u_k)\subset \hat{\Gamma}_1(v_0,w_0)$, by Lemma \ref{lem:2.50},
there is a $U^*\in \hat{\Gamma}_1(v_0,w_0)$ such that up to a subsequence $u_k\to U^*$ pointwise as $k\to\infty$,
\begin{equation}\label{eq:6.31}
  -K\leq J_1(U^*)\leq M+1+2K_1,
\end{equation}
and
\begin{equation}\label{eq:6.32}
  \norm{U^*-\phi}_{X_i}\geq \sigma
\end{equation}
for all $i\in\Z$ and $\phi\in\{v_0,w_0\}$.

Take $U\in\MM_1(v_0,w_0)$ such that
\begin{equation}\label{eq:6.33}
  \norm{U-w_0}_{X_i}\leq \frac{\sigma}{6} \quad \textrm{ for any } i\geq 0.
\end{equation}
This is possible, if not, replacing $U$ by $\tau_{-j}^{1}U$ for large $j\in\N$.
Define
\begin{equation*}
  \mathcal{B}:=\{\tau_{-j}^{1}U^*\,|\,j\in\Z\},
\end{equation*}
\begin{equation*}
  \begin{split}
  \MY:=\{u\in\hat{\Gamma}_1(v_0,w_0)\,|\, &u\leq U \textrm{ and } \norm{u-g}_{\BD{T}_i}\to 0, \textrm{ as } i\to \infty \\
        & \textrm{ for some } g=g(u)\in\mathcal{B}\},
  \end{split}
\end{equation*}
and
\begin{equation}\label{eq:6.34}
  c_1(\MY):=\inf_{u\in\MY}J_1(u).
\end{equation}
To show $c_1(\MY)\in\R$, set
\begin{equation*}
  f:=\left\{
      \begin{array}{ll}
        v_0, & \BD{i}_1\leq 0, \\
        \min(U,U^*), &  \BD{i}_1\geq 1.
      \end{array}
    \right.
\end{equation*}
Obviously, $f\in\MY \neq \emptyset$.
For $1+r\leq p\leq q$, by Lemma \ref{lem:2.6miao},
\begin{equation*}
  J_{1;p,q}(f)+J_{1;p,q}(\max(U,U^*))\leq J_{1;p,q}(U)+J_{1;p,q}(U^*).
\end{equation*}
Thus %
\begin{equation*}
\begin{split}
 J_{1;p,q}(f) \leq& J_{1;p,q}(U)+J_{1;p,q}(U^*)-J_{1;p,q}(\max(U,U^*))\\
  \leq& (J_1(U)+2K_1)+(J_1(U^*)+2K_1)+K_1 \\
  \leq& (c_1(v_0,w_0)+2K_1)+(M+1+2K_1+K_1)+K_1 \\
  <&\infty,
  \end{split}
\end{equation*}
where the first and the second inequalities follow from Lemmas \ref{lem:2.6miao} and \ref{lem:23}, respectively; the third inequality follows from the choice of $U$ and \eqref{eq:6.31}.
So by Lemma \ref{lem:23},
\[-K_1\leq c_1(\MY)\leq J_1(f)<\infty.\]

The idea to obtain a contradiction is that we can use variational problem \eqref{eq:6.34} to construct a solution of \eqref{eq:PDE}, say $\Phi$.
Then choose $W\in \MM_1(v_0,w_0)$ such that $\Phi$ and $W$
are `cross', that is, there are $\BD{i},\BD{j}\in\Z^n$, with $\BD{i}\neq\BD{j}$, such that $\Phi(\BD{i})\leq W(\BD{i})$, $\Phi(\BD{j})>W(\BD{j})$.
The contradiction lies in that we can prove $\min(\Phi,W)$ also is a solution of \eqref{eq:PDE}, contrary to Corollary \ref{cor:bijiao}.

Now
take a minimizing sequence, say $(\phi_k)$, for \eqref{eq:6.34} satisfying $J_1(\phi_k)\leq c_1(\MY)+1$. Then for any $k\in\N$, there exist an $s_k\in\N$ and a $g_k\in\mathcal{B}$ such that for $s\geq s_k$,
\begin{equation}\label{eq:6.36}
  \norm{\phi_k-g_k}_{X_s}\leq \frac{\sigma}{6}.
\end{equation}
To compare the desired solutions $\Phi,W$, we ask $\Phi$, obtained by minimization method, satisfies some uniform condition (see \eqref{eq:6.45} below).
Thus translating $\phi_k$ by $\tau_{-s_k}^{1}$, but $\tau_{-s_k}^{1}\phi_k$ may be not contained in $\MY$.
To overcome this difficulty, we truncate $\tau_{-s_k}^{1}\phi_k$ obtaining another minimizing sequence for \eqref{eq:6.34}.
Set $\psi_k=\max(\tau_{-s_k}^{1}\phi_k, U)$ and $\chi_k=\min(\tau_{-s_k}^{1}\phi_k, U)$.
It is easy to see that $\psi_k\in \Gamma_1(v_0,w_0)$ and $\chi_k\in \hat{\Gamma}_1(v_0,w_0)$, $\chi_k\leq U$.
Noticing that
\begin{equation*}
  \norm{\chi_k-\tau_{-s_k}^{1}g_k}_{\BD{T}_i}\leq \left\{
                                                    \begin{array}{ll}
                                                      \norm{\tau_{-s_k}^{1}\phi_k-\tau_{-s_k}^{1}g_k}_{\BD{T}_i}, & \textrm{ if } \tau_{-s_k}^{1}\phi_k (\BD{T}_i)\leq U(\BD{T}_i) , \\
                                                      \norm{\tau_{-s_k}^{1}\phi_k-\tau_{-s_k}^{1}g_k}_{\BD{T}_i}, & \textrm{ if } \tau_{-s_k}^{1}g_k (\BD{T}_i)\leq U(\BD{T}_i)< \tau_{-s_k}^{1}\phi_k (\BD{T}_i),\\
                                                      \norm{U-w_0}_{\BD{T}_i}, & \textrm{ if } U(\BD{T}_i)< \min\{\tau_{-s_k}^{1}g_k (\BD{T}_i), \tau_{-s_k}^{1}\phi_k (\BD{T}_i)\},
                                                    \end{array}
                                                  \right.
\end{equation*}
we have
\begin{equation}\label{eq:6.38}
  \norm{\chi_k-\tau^{1}_{-s_k}g_k}_{\BD{T}_i}\leq 2\norm{\tau_{-s_k}^{1}\phi_k-\tau_{-s_k}^{1}g_k}_{\BD{T}_i}+\norm{U-w_0}_{\BD{T}_i}\to 0
\end{equation}
as $i\to\infty$.
Thus $\chi_k\in\MY$.

We claim:
\begin{equation}\label{eq:6.39}
  J_1(\psi_k)+J_1(\chi_k)\leq J_1(\phi_k)+J_1(U).
\end{equation}
Note we can not take limit in
Lemma \ref{lem:2.6miao} since $\chi_k\not\in \Gamma_1(v_0,w_0)$. We need caution to prove \eqref{eq:6.39}.
By Lemma \ref{lem:2.6miao}, for any $p<q\in\Z$,
\begin{equation}\label{eq:6.40}
  J_{1;p,q}(\psi_k)+J_{1;p,q}(\chi_k)\leq J_{1;p,q}(\tau_{-s_k}^{1}\phi_k)+J_{1;p,q}(U).
\end{equation}
By Lemma \ref{lem:23},
\begin{equation}\label{eq:jia}
\begin{split}
  &J_{1;p,q}(\tau_{-s_k}^{1}\phi_k)\\
\leq & J_1(\tau_{-s_k}^{1}\phi_k)+2K_1\\
=& J_1(\phi_k)+2K_1\\
\leq & c_1(\MY)+1+2K_1.
\end{split}
\end{equation}
Therefore \eqref{eq:6.40} and \eqref{eq:jia}
imply $J_1(\psi_k), J_1(\chi_k)<\infty$.
Taking $p_i\to -\infty, q_i\to \infty$ as $i\to\infty$ such that $J_{1;p_i,q_i}(\tau_{-s_k}^{1}\phi_k)\to J_1(\tau_{-s_k}^{1}\phi_k)=J_1(\phi_k)$ as $i\to \infty$,
Lemma \ref{lem:23} gives
\begin{equation*}
\begin{split}
  J_1(\chi_k)\leq & \liminf_{i\to \infty}J_{1;p_i,q_i}(\chi_k)\\
\leq & J_1 (\phi_k)+J_1(U)-J_1(\psi_k),
\end{split}
\end{equation*}
i.e., \eqref{eq:6.39} holds.

Since $\psi_k\in\Gamma_1(v_0,w_0)$, $J_1(U)=c_1(v_0,w_0)\leq J_1(\psi_k)$.
Thus by \eqref{eq:6.39},
\begin{equation}\label{eq:6.44}
  J_1(\chi_k)\leq J_1(\phi_k).
\end{equation}
Hence $(\chi_k)$ is a modified minimizing sequence of \eqref{eq:6.34}.
Thus by Lemma \ref{lem:2.50},
there is a $\Phi\in\hat{\Gamma}_1(v_0,w_0)$ such that $\chi_k\to\Phi$ (up to a subsequence) pointwise as $k\to\infty$.
As in \eqref{eq:claim}, $\Phi$ is a solution of \eqref{eq:PDE}.
Moreover, $\Phi>v_0$.
If not, by Lemma \ref{lem:unknown} and $\Phi\geq v_0$, $\Phi=v_0$.
Thus $\chi_k=\min(\tau_{-s_k}^{1}\phi_k,U)\to v_0$ pointwise as $k\to \infty$.
But $U>v_0$, so for all large $k$,
\begin{equation}\label{eq:6.51}
\norm{\tau_{-s_k}^{1}\phi_k-v_0}_{X_0}=\norm{\phi_k-v_0}_{X_{s_k}}\leq  \frac{\sigma}{3}.
\end{equation}
Hence for all large $k$, one have the following contradiction:
\[
\sigma  \leq \norm{g_k -v_0}_{X_{s_k}}  \leq \norm{g_k-\phi_k}_{X_{s_k}}+\norm{\phi_k-v_0}_{X_{s_k}}\leq \frac{\sigma}{6}+\frac{\sigma}{3} ,
\]
where the first inequality follows from \eqref{eq:6.32} and the last inequality follows from \eqref{eq:6.36}, \eqref{eq:6.51}.
Thus $\Phi>v_0$.

We claim that $\Phi$ satisfies the following uniform condition: 
\begin{equation}\label{eq:6.45}
  \norm{\Phi-w_0}_{X_i}\geq \frac{\sigma}{2} \quad \quad \quad \textrm{for any $i\geq 0$}.
\end{equation}
Indeed,
\begin{equation}\label{eq:6.47}
\begin{split}
  &\norm{\chi_k-w_0}_{X_i} \\
 \geq  & \norm{w_0-\tau_{-s_k}^{1}g_k}_{X_i}-\norm{\chi_k-\tau_{-s_k}^{1}g_k}_{X_i}\\
\geq & \sigma-\norm{\chi_k-\tau_{-s_k}^{1}g_k}_{X_i}\\
\geq & \sigma- \Big[\sum_{\BD{j}\in X_i\cap \{U\geq \tau_{-s_k}^{1}\phi_k\}}\norm{\tau_{-s_k}^{1}\phi_k-\tau_{-s_k}^{1}g_k}_{\BD{j}}\\
&\quad\quad +\sum_{\BD{j}\in X_i\cap \{\tau_{-s_k}^{1}g_k\leq U< \tau_{-s_k}^{1}\phi_k\}} \norm{\tau_{-s_k}^{1}\phi_k-\tau_{-s_k}^{1}g_k}_{\BD{j}}
+\sum_{\BD{j}\in X_i\cap \{U<\min(\tau_{-s_k}^{1}g_k, \tau_{-s_k}^{1}\phi_k)\}}\norm{U-w_0}_{\BD{j}}\Big]\\
\geq & \sigma -(\frac{\sigma}{6}+\frac{\sigma}{6}) \\
\geq &\frac{\sigma}{2},
\end{split}
\end{equation}
where the third inequality is implied similarly to \eqref{eq:6.38} and the fourth inequality follows from \eqref{eq:6.36}, \eqref{eq:6.33}.
Letting $k\to \infty$ gives \eqref{eq:6.45}.

Since $\Phi>v_0$,
we can take $W\in\MM_1(v_0,w_0)$ such that $W<\Phi$ on $X_0$.
Noting by \cite[Theorem 3.13]{LC}, $W\in\MM_1(v_0,w_0)$ implies $W$ is a solution of \eqref{eq:PDE}.
Now we prove $\min(W, \Phi)$ also is a solution of \eqref{eq:PDE}.
As before, one can prove
$\min(W,\chi_k)\in\MY$ and $\max(W,\chi_k)\in\Gamma_1(v_0,w_0)$. Therefore as in \eqref{eq:6.44}, 
\begin{equation*}
  J_1(\min(W,\chi_k))\leq J_1(\chi_k),
\end{equation*}
and as $k\to \infty$, $\min(W,\chi_k)$ converges pointwise to $\min(W,\Phi)$, a solution of \eqref{eq:PDE}. 
Note that for $i\in\Z$ large enough, by \eqref{eq:6.45}, $W(\BD{T}_i)> \Phi(\BD{T}_i)$. 
Since $\Phi, W$ and $\min(W,\Phi)$ are solutions of \eqref{eq:PDE},
as stated earlier, we obtain a contradiction by Corollary \ref{cor:bijiao}.
\end{proof}

If $u$ in Proposition \ref{prop:6.27} is a solution of \eqref{eq:PDE}, it should be asymptotic to periodic solutions.
This fact is proved in the
following proposition.
\begin{prop}\label{prop:6.53}
Suppose \eqref{eq:*0} holds and $u\in\hat{\Gamma}_1(v_0,w_0)$ with $J_1(u)\leq M<\infty$.
If $u$ satisfies \eqref{eq:PDE}
for $\BD{i}_1\geq R$ (resp. $\BD{i}_1\leq -R$), then
$\norm{u-\phi}_{X_i}\to 0$ as $i\to \infty$ (resp. $\norm{u-\phi}_{X_i}\to 0$ as $i\to -\infty)$,
where $R\in\R$ and $\phi=v_0$ or $w_0$.
\end{prop}
\begin{proof}
By Proposition \ref{prop:6.27}, for any $\sigma>0$, we obtain sequences $(t_k)\subset \Z$, 
$(s_k)=(s_k(\sigma)\subset \N$ with $t_k\to \infty$ and $s_{k+1}-s_k\geq 2r+1$, and $\phi\in\{v_0,w_0\}$
such that
\begin{equation}\label{eq:6.55}
  \norm{u-\phi}_{X_{s_k}}\leq \sigma.
\end{equation}
For such a $\phi$, we shall prove
\begin{equation}\label{eq:6.56}
  \norm{u-\phi}_{X_i}\to 0,
\end{equation}
as $i\to \infty$.
Suppose, by contradiction, \eqref{eq:6.56} fails,
then there are a $\gamma>0$ and a sequence $p_i\to\infty$ as $i\to \infty$ such that
\begin{equation}\label{eq:6.59}
  \norm{u-\phi}_{X_{p_i}}\geq \gamma.
\end{equation}
Deleting some $s_k$ if necessary, we can assume that $p_i\in (s_i+r,s_{i+1}-r)$. Define
\begin{equation*}
  h_i=\left\{
        \begin{array}{ll}
          \phi, & \BD{i}_1 \leq s_i-r-1 \quad \textrm{ or } \quad \BD{i}_1\geq s_{i+1}+r+1, \\
          u, & s_i-r\leq \BD{i}_1 \leq s_{i+1}+r,
        \end{array}
      \right.
\end{equation*}
then $h_i\in \Gamma_1(\phi)\cap \hat{\Gamma}_1(v_0,w_0)$.
Choosing $\sigma>0$ small enough, we have
\begin{equation*}
  J_{1;s_i-2r,s_i-1}(h_i)+J_{1;s_{i+1},s_{i+1}+2r+1}(h_i)\leq 2\kappa(\sigma) \leq \frac{2}{3}\beta(\gamma) ,
\end{equation*}
where $\kappa(\theta)\to 0$ as $\theta\to 0$ and $\beta(\gamma)$ is given by Proposition \ref{prop:6.13}.
Thus
\begin{equation}\label{eq:6.66}
  \begin{split}
     & J_{1;s_i,s_{i+1}-1}(u)  \\
     =&J_1(h_i)-J_{1;s_i-2r,s_i-1}(h_i)-J_{1;s_{i+1},s_{i+1}+2r+1}(h_i)\\
     \geq & \beta(\gamma)-J_{1;s_i-2r,s_i-1}(h_i)-J_{1;s_{i+1},s_{i+1}+2r+1}(h_i)\\
     \geq & \beta(\gamma)-\frac{2}{3}\beta(\gamma)\\
      =& \frac{1}{3}\beta(\gamma).
  \end{split}
\end{equation}

Suppose that $s_i>R+r$ for $i\geq i_0$.
We have the following contradiction:
\begin{equation*}
\begin{split}
  M\geq &J_1(u) \\
  =&J_{1;-\infty,s_{i_0 -1}}(u)+\sum_{j=0}^{q-1}J_{1;s_{i_0+j}, s_{i_0+j+1}-1}(u)+J_{1;s_{i_0}+q,\infty}(u)\\
  \geq & -K_1 +\sum_{j=0}^{q-1}J_{1;s_{i_0+j}, s_{i_0+j+1}-1}(u)-K_1  \\
  \geq & -2K_1 +\frac{q}{3}\beta(\gamma) \\
  \to & \infty  \quad\textrm{as } q\to \infty,
  \end{split}
\end{equation*}
where the third inequality follows from Lemma \ref{lem:23} and the last inequality follows from \eqref{eq:6.66}.
Thus \eqref{eq:6.56} holds, and this proves Proposition \ref{prop:6.53}.
\end{proof}

We also need the following comparison result, which is useful in proving that the minima of $J_1$ over $Y_{\BD{m},l}$ is a solution of \eqref{eq:PDE}.
Recalling
$\rho_i$ is defined in \eqref{eq:6.2}, we set
\begin{equation*}
\begin{split}
  \Lambda_1(v_0,w_0)&=\{u\in\Gamma_1(v_0,w_0)\,|\, \norm{u-v_0}_{\BD{T}_0}=\rho_1 \textrm{ or } \norm{u-w_0}_{\BD{T}_0}=\rho_2\}\\
  (resp. \quad \Lambda_1(w_0,v_0)&=\{u\in\Gamma_1(w_0,v_0)\,|\, \norm{u-v_0}_{\BD{T}_0}=\rho_4 \textrm{ or } \norm{u-w_0}_{\BD{T}_0}=\rho_3\})
  \end{split}
\end{equation*}
and
\begin{eqnarray}
  d_1(v_0,w_0) &=& \inf_{u\in\Lambda_1(v_0,w_0)}J_1(u), \label{eq:6.73}\\
  (resp.  \quad  d_1(w_0,v_0) &=& \inf_{u\in\Lambda_1(w_0,v_0)}J_1(u)). \nonumber
\end{eqnarray}

\begin{prop}\label{prop:6.74}
With $d_1(v_0,w_0)$ (resp. $d_1(w_0,v_0)$) defined as in \eqref{eq:6.73}, we have
\[d_1(v_0,w_0)>c_1(v_0,w_0) \quad\quad(resp. \quad d_1(w_0,v_0)>c_1(w_0,v_0)).\]
\end{prop}

\begin{proof}
$d_1(v_0,w_0)\geq c_1(v_0,w_0)$ is obvious since $\Lambda_1(v_0,w_0)\subset \Gamma_1(v_0,w_0)$.
To prove the strict inequality, take a minimizing sequence $(u_k)$ for \eqref{eq:6.73}.
Thus
\begin{equation}\label{eq:123000}
  \norm{u_k-v_0}_{\BD{T}_0}=\rho_1 \quad\quad\textrm{or}\quad\quad \norm{u_k-w_0}_{\BD{T}_0}=\rho_2.
\end{equation}
By Lemmas \ref{lem:2.50} and \ref{lem:2.64}, there is a $P\in \hat{\Gamma}_1(v_0,w_0)$ with $J_1(P)<\infty$
such that $u_k\to P$ (taking a subsequence if necessary) pointwise as $k\to \infty$, and
\begin{equation}\label{eq:6.76}
  \norm{P-v_0}_{\BD{T}_0}=\rho_1 \quad \textrm{ or } \quad \norm{P-w_0}_{\BD{T}_0}=\rho_2.
\end{equation}
Proceeding as in \eqref{eq:claim} shows that $P$ is a solution of \eqref{eq:PDE} whenever $\BD{i}_1\neq 0$.
Thus by Proposition \ref{prop:6.53}, for some $\phi,\psi\in \{v_0,w_0\}$, we obtain 
\begin{equation}\label{eq:6.77}
  \norm{P-\phi}_{X_i}\to 0 \quad\quad \textrm{and}\quad\quad \norm{P-\psi}_{X_{-i}}\to 0
\end{equation}
as $i\to \infty$.
Now our discussions are divided into three cases: (i) $\psi=w_0$, or (ii) $\phi=v_0$, or (iii) $\psi=v_0$ and $\phi=w_0$.

Suppose $\psi=w_0$. For any $\epsilon>0$, by \eqref{eq:6.77}, there is an $s\in \{i\in\Z \,|\, i<-r\}$ such that for any $k\geq k_0(s)$,
\begin{equation*}
  \norm{u_k-w_0}_{X_s}\leq \epsilon.
\end{equation*}
Noting $u_k\in\Gamma_1(v_0,w_0)$, for any $k\in\N$, there exists $q=q(k)\in\N$ large enough, such that
\begin{equation*}
  \norm{u_k-w_0}_{X_q}\leq \epsilon.
\end{equation*}
Define
\begin{equation*}
g_k=\left\{
      \begin{array}{ll}
        w_0, & \BD{i}_1\leq s-1 \quad\textrm{ or }\quad \BD{i}_1\geq q+1, \\
        u_k, & s\leq \BD{i}_1 \leq q,
      \end{array}
    \right.
\end{equation*}
and
\begin{equation*}
h_k=\left\{
      \begin{array}{ll}
       u_k, & \BD{i}_1\leq s-1 \quad \textrm{ or }\quad \BD{i}_1\geq q+1, \\
        w_0, & s\leq \BD{i}_1 \leq q.
      \end{array}
    \right.
\end{equation*}
Then we have
\begin{equation}\label{eq:6.84}
  J_1(u_k)\geq J_1(g_k)+J_1(h_k)-\kappa(\epsilon),
\end{equation}
where $\kappa(\theta)\to 0$ as $\theta\to 0$,
and $g_k\in\Gamma_1(w_0)$, and
\begin{equation*}
  \norm{g_k-w_0}_{\BD{T}_0}=\norm{u_k-w_0}_{\BD{T}_0}.
\end{equation*}
Thus by \eqref{eq:123000},
\begin{equation*}
  \norm{w_0-g_k}_{\BD{T}_0}\geq \min(\rho_2,\bar{\rho}-\rho_1)=: \gamma.
\end{equation*}
Applying Proposition \ref{prop:6.13} gives
\begin{equation}\label{eq:6.89}
  J_1(g_k)\geq \beta(\gamma).
\end{equation}
Since $h_k\in\Gamma_1(v_0,w_0)$, by \eqref{eq:6.84} and \eqref{eq:6.89},
\begin{equation}\label{eq:6.90}
  J_1(u_k)\geq \beta(\gamma)+c_1(v_0,w_0)-\kappa(\epsilon).
\end{equation}
Taking $\epsilon$ small enough such that
\begin{equation*}
  2\kappa(\epsilon)\leq \beta(\gamma),
\end{equation*}
and then letting $k\to \infty$ in \eqref{eq:6.90} shows
\begin{equation}\label{eq:6.92}
  d_1(v_0,w_0)\geq c_1(v_0,w_0)+\frac{1}{2}\beta(\gamma).
\end{equation}

If case (ii), i.e., $\phi=v_0$ occurs, the above argument can be easily modified to prove that \eqref{eq:6.92} still holds.

Now suppose (iii) is satisfied, i.e., $\psi=v_0$ and $\phi=w_0$.
By \eqref{eq:6.76} and \eqref{eq:6.77}, $P\in\Lambda_1(v_0,w_0)$ and then $J_1(P)\geq d_1(v_0,w_0)$.
We claim that
\begin{equation}\label{eq:claim789541}
  d_1(v_0,w_0)\geq J_1(P).
\end{equation}
Suppose \eqref{eq:claim789541} holds for the moment.
If $d_1(v_0,w_0)=c_1(v_0,w_0)$, then $P\in \MM_1(v_0,w_0)$ follows from $P\in\Gamma_1(v_0,w_0)$.
But this contradicts \eqref{eq:6.76} and the choices of $\rho_1, \rho_2$.
Thus $d_1(v_0,w_0)> c_1(v_0,w_0)$.

To complete the proof of Proposition \ref{prop:6.74}, we need to prove \eqref{eq:claim789541}.
For this purpose, we use the argument as in the proof of (C) of \cite[Theorem 3.13]{LC}.
Define $\hat{\BD{T}}_i=\cup_{j=i-r-1}^{i+r+1}\BD{T}_j$. For any $\epsilon >0$, by \eqref{eq:6.6}, there is a $p_0=p_0(\epsilon)$ such that if $p\geq p_0$,
\begin{equation*}
    \norm{P-v_0}_{\hat{\BD{T}}_{-p}}\leq \epsilon/2, \quad \norm{P-w_0}_{\hat{\BD{T}}_{p}}\leq \epsilon/2.
\end{equation*}
Since $u_k\to P$ as $k\to \infty$, then
for any $p\geq p_0$, there is a $k_0=k_0(p)$ such that for any $k\geq k_0$,
\begin{equation*}
    \norm{u_k-P}_{\hat{\BD{T}}_{-p}}\leq \epsilon/2, \quad  \norm{u_k-P}_{\hat{\BD{T}}_{p}}\leq \epsilon/2.
\end{equation*}
Thus for such $k$ and $p$,
\begin{equation}\label{eq:3.19}
    \norm{u_k-v_0}_{\hat{\BD{T}}_{-p}}\leq \epsilon, \quad \norm{u_k-w_0}_{\hat{\BD{T}}_{p}}\leq \epsilon.
\end{equation}
For any fixed $k\geq k_0(p)$, since $u_k\in \Lambda_1(v_0,w_0)$, there is a $q_0=q_0(k)$ such that for $q\geq q_0$,
\begin{equation}\label{eq:3.21}
    \norm{u_k-v_0}_{\hat{\BD{T}}_{-q}}\leq \epsilon, \quad \norm{u_k-w_0}_{\hat{\BD{T}}_{q}}\leq \epsilon.
\end{equation}
Define
\begin{equation*}
    \bar{f}_k=\left\{
          \begin{array}{ll}
            w_0, & p-r\leq \BD{i}_1 \leq p+r \quad \textrm{ or }\quad  q-r\leq \BD{i}_1 \leq q+r,\\
            u_k, & p+r+1\leq \BD{i}_1\leq q-r-1,
          \end{array}
        \right.
\end{equation*}
and
\begin{equation*}
    \bar{g}_k=\left\{
          \begin{array}{ll}
            v_0, & -q-r\leq \BD{i}_1 \leq -q+r \quad \textrm{ or } \quad -p-r\leq \BD{i}_1 \leq -p+r,\\
            u_k, & -q-r+1\leq \BD{i}_1\leq -p+r-1,
          \end{array}
        \right.
\end{equation*}
Now we extend $\bar{f}_k$ (resp. $\bar{g}_k$) to a $(q+2r+1-p)$-periodic function of $\BD{i}_1$ and still denote it by $\bar{f}_k$ (resp. $\bar{g}_k$).
Then by \eqref{eq:3.19}-\eqref{eq:3.21}, there is a $\kappa_1(\epsilon)$ ($\kappa_1(\epsilon)\to 0$ as $\epsilon\to 0$) such that
\begin{equation}\label{eq:3.23}
\begin{split}
    |J_{1;p,q}(u_k)-J_{1;p,q}(\bar{f}_k)|&\leq \kappa_1(\epsilon),\\
    |J_{1;-q,-p}(u_k)-J_{1;-q,-p}(\bar{g}_k)|&\leq \kappa_1(\epsilon).
    \end{split}
\end{equation}
By \cite[Proposition 3.1]{LC},
\begin{equation}\label{eq:3.24}
\begin{split}
    J_{1;p,q}(\bar{f}_k)&=J_{1;p-r,q+r}(\bar{f}_k)\geq 0,\\
    J_{1;-q,-p}(\bar{g}_k)&=J_{1;-q-r,-p+r}(\bar{g}_k)\geq 0.
    \end{split}
\end{equation}
By \eqref{eq:3.23}-\eqref{eq:3.24},
\begin{equation*}
\begin{split}
    J_{1;1,\infty}(u_k)&=J_{1;1,p-1}(u_k)+J_{1;p,q}(u_k)+J_{1;q+1,\infty}(u_k)\\
    &\geq J_{1;1,p-1}(u_k)-\kappa_1(\epsilon)+J_{1;q+1,\infty}(u_k),\\
    J_{1;-\infty,0}(u_k)&=J_{1;-\infty,-q-1}(u_k)+J_{1;-q,-p}(u_k)+J_{1;-p+1,0}(u_k)\\
    &\geq J_{1;-\infty,-q-1}(u_k)-\kappa_1(\epsilon)+J_{1;-p+1,0}(u_k).
    \end{split}
\end{equation*}
Adding the above two inequalities and
letting $q\to \infty$, we get
\begin{equation*}
    J_1(u_k)\geq J_{1;-p+1,p-1}(u_k)-2\kappa_1(\epsilon).
\end{equation*}
Thus letting $k\to \infty$ shows that
\begin{equation*}
    d_1(v_0,w_0)\geq J_{1;-p+1,p-1}(P)-2\kappa_1(\epsilon).
\end{equation*}
Finally,
letting $p\to\infty$ and then $\epsilon\to 0$ yields
\begin{equation*}
    d_1(v_0,w_0)\geq J_1(P).
\end{equation*}
This proves \eqref{eq:claim789541} and complete the proof of $d_1(v_0,w_0)>c_1(v_0,w_0)$.
$d_1(w_0,v_0)>c_1(w_0,v_0)$ can be proved similarly.
\end{proof}

The following result means that minimal solutions of \eqref{eq:PDE} in $\hat{\Gamma}(v_0,w_0)$ are Birkhoff.
It will be used to show that our solution obtained in Theorem \ref{thm:6.8} is not minimal any more.

\begin{prop}\label{prop:6.93}
If $u\in\hat{\Gamma}_1(v_0,w_0)$ is minimal, $u$ is Birkhoff.
\end{prop}
The proof of Proposition \ref{prop:6.93} almost follows \cite[Proposition 6.93]{RS} with slight modifications.
For example,
in the proof of \cite[Proposition 6.93]{RS}, the quoted Corollary 6.54 and Theorem 3.2 are replaced by
Proposition \ref{prop:6.53} in the present paper
and \cite[Theorem 3.13]{LC}, respectively;
$J_{1,-p-1}(u_p)$ and $J_{1,p}(u_p)$ (\cite[p.79, line -4]{RS}) are replaced by $J_{1;-p-1-r,-p+r}(u_p)$ and $J_{1;p-r,p+1+r}(u_p)$ since our problem is nonlocal.
Thus we omit the proof here.

\section{Proof of Theorem \ref{thm:6.8}}\label{chap:7}
Now we prove Theorem \ref{thm:6.8}.
Take a minimizing sequence $(u_k)$ for \eqref{eq:6.7}.
We claim that
there is an $M>0$ such that
$J_1(u_k)\leq M$
for all $k\in\N$.
Indeed, fix some $V_1\in\MM_1(v_0,w_0)$ and $W_1\in\MM_1(w_0,v_0)$. Translating $V_1$ (resp. $W_1$) by $\tau^1$ if necessary, we may assume that $V_1$ (resp. $W_1$) satisfies
\begin{equation*}
\begin{split}
  &\rho_{-}(\tau^{1}_{-i}V_1)\leq \rho_1, \quad \textrm{for any } i\leq \BD{m}_1,\\
  (\textrm{resp. }\quad &\rho_{-}(\tau^{1}_{-i}W_1)\leq \rho_4, \quad \textrm{for any } i \geq \BD{m}_4)
  \end{split}
\end{equation*}
and $J_{1;p,q }(V_1)\leq 1/2$ (resp. $J_{1;p,q}(W_1)\leq 1/2$) for any $p\leq q\leq \BD{m}_1$ (resp. $q\geq p\geq \BD{m}_4$).
Let
\begin{equation}\label{eq:kappa}
  \kappa(\theta):=\sup_{u:\sum^{r}_{j=-r}\norm{u-v_0}_{\BD{T}_j}\leq \theta}J_{1,0}(u)+\sup_{v:\sum^{r}_{j=-r}\norm{v-w_0}_{\BD{T}_j}\leq \theta}J_{1,0}(v).
\end{equation}
Set
\begin{equation*}
  \hat{U}=\left\{
            \begin{array}{ll}
             V_1, & \BD{i}_1\leq \BD{m}_1, \\
              w_0, & \BD{m}_1+1\leq \BD{i}_1\leq \BD{m}_4-1, \\
              W_1, & \BD{m}_4\leq \BD{i}_1.
            \end{array}
          \right.
\end{equation*}
If $\BD{m}_4-\BD{m}_1\geq 2r+1$,
\begin{equation*}
\begin{split}
  J_1(\hat{U})&=J_{1;-\infty,\BD{m_1}-r-1}(V_1) +J_{1;\BD{m}_1-r,\BD{m}_4+r}(\hat{U})+ J_{1;\BD{m}_4+r+1,\infty}(W_1)\\
  &=J_{1;-\infty,\BD{m_1}-r-1}(V_1) +J_{1;\BD{m}_1-r,\BD{m}_1+r}(\hat{U})+J_{1;\BD{m}_4-r,\BD{m}_4+r}(\hat{U})+ J_{1;\BD{m}_4+r+1,\infty}(W_1)\\
  &\leq \frac{1}{2}+\kappa(\bar{\rho})(2r+1+2r+1)+\frac{1}{2}\\
  &=:M;
  \end{split}
\end{equation*}
if $\BD{m}_4-\BD{m}_1\leq 2r$,
\begin{equation}\label{eq:mm}
\begin{split}
  J_1(\hat{U})&=J_{1;-\infty,\BD{m_1}-r-1}(V_1) +J_{1;\BD{m}_1-r,\BD{m}_4+r}(\hat{U})+ J_{1;\BD{m}_4+r+1,\infty}(W_1)\\
  &\leq \frac{1}{2} +\kappa(\bar{\rho})(\BD{m}_4-\BD{m}_1+2r+1)+\frac{1}{2}\\
  &\leq M.
  \end{split}
\end{equation}
Note that $M>0$ is independent of $\BD{m}$ and $l$.

By Lemma \ref{lem:2.50}, with $\MY=Y_{\BD{m},l}$,
there is a $U\in\hat{\Gamma}_1(v_0,w_0)$ satisfying
$u_k\to U$ (maybe up to a subsequence) pointwise as $k\to \infty$, and such that
\begin{equation*}
  J_1(U)\leq M+2K_1
\end{equation*}
and $U$ satisfies \eqref{eq:6.5}.
$U$ will be shown to be a solution of \eqref{eq:PDE}.
Note that using the arguments of proving \eqref{eq:claim}, $U$ satisfies \eqref{eq:PDE} outside the four constraint regions of \eqref{eq:6.5}.

To complete the proof of Theorem \ref{thm:6.8}, we shall prove:\\
(A) if $l\gg 0$, there is an $X_i$ in every constraint region such that $U$ satisfies \eqref{eq:PDE} on $X_i$; \\
(B) $U$ satisfies \eqref{eq:6.6} and then $U\in Y_{\BD{m},l}$; \\
(C) $J_1(U)=b_{\BD{m},l}$; \\
(D) if $\BD{m}_2-\BD{m}_1$, $\BD{m}_4-\BD{m}_3\gg 0$, $U$ satisfies \eqref{eq:PDE} on the four constraint regions.\\

\emph{Proof of (A).}
The proof of (A) almost the same as `Proof of (A)' in \cite[p.82]{RS}.
Here we provide the details for the reader's convenience.
Applying Proposition \ref{prop:6.27} to \[\sigma\in (0, \min_{1\leq j\leq 4}(\rho_j,\bar{\rho}-\rho_j))\] and $M$ gives an $l_0=l_0(\sigma,M)$.
For $l\geq l_0$, we have an $X_i\subset [\BD{m}_1-l,\BD{m}_1]\times \{0\}^{n-1}$ and $\phi_i\in\{v_0,w_0\}$ such that
\begin{equation*}
  \norm{U-\phi_i}_{X_i}\leq \sigma.
\end{equation*}
Owing to the choice of $\sigma$, $\phi_i=v_0$ and
\begin{equation}\label{eq:1221221}
  \rho_-(\tau_{-j}^{1}U)<\rho_1
\end{equation}
for all $j\in X_i$.
We claim that
\begin{equation}\label{eq:claim11212}
  \textrm{$U$ satisfies \eqref{eq:PDE} on $X_i$.}
\end{equation}
In fact, note that
\eqref{eq:1221221} also holds for $u_k$ with $k$ large enough.
So for $t$ small enough and $k$ large, we have $\max(u_k+t\delta_{\BD{T}_j},v_0)\in Y_{\BD{m},l}$ and
\begin{equation*}
\begin{split}
 b_{\BD{m},l}\leq  & J_1(u_k)=:  \epsilon' _k+ b_{\BD{m},l} \\
   \leq &\epsilon' _k+ J_1(\max(u_k+t\delta_{\BD{T}_j},v_0))\\
    \leq& \epsilon' _k+J_1(\max(u_k+t\delta_{\BD{T}_j},v_0))+J_1(\min(u_k+t\delta_{\BD{T}_j},v_0)) \\
    \leq &\epsilon' _k+J_1(u_k+t\delta_{\BD{T}_j}) ,
\end{split}
\end{equation*}
where $\epsilon' _k\to 0$ as $k\to \infty$. The fourth inequality follows from Lemma \ref{lem:2.72} and $\min(u_k+t\delta_{\BD{T}_j},v_0)\in \Gamma_1(v_0)$.
The last inequality follows form Lemma \ref{lem:2.6miao}.
Now Lemma \ref{lem:2.64} implies \eqref{eq:claim11212}.

Similarly, one can obtain three subsets (we abuse notations here by a same notation) $X_i$ %
contained in $[\BD{m}_2,\BD{m}_2+l]\times \{0\}^{n-1}$, $[\BD{m}_3-l,\BD{m}_3]\times \{0\}^{n-1}$ and $[\BD{m}_4,\BD{m}_4+l]\times \{0\}^{n-1}$,
with corresponding $\phi_i=w_0$ in the two former cases and $\phi_i=v_0$ in the third case.
That $U$ satisfies \eqref{eq:PDE} on these $X_i$ can be proved as that of \eqref{eq:claim11212} with a few obvious modifications.
\\

\emph{Proof of (B).}
We only check the case of $i\to\infty$ since the other case can be proved similarly.
By Proposition \ref{prop:6.53} with $R=\BD{m}_4+l$, we have
\begin{equation}\label{eq:7.6}
  \norm{U-\phi}_{X_j}\to 0, \quad j\to \infty,
\end{equation}
for some $\phi\in\{v_0,w_0\}$.
If $\phi=v_0$, we are done.
Now suppose, by contradiction,
$\phi=w_0$.
By \eqref{eq:7.6}, for large $k$ we have
\begin{equation}\label{eq:7.8}
\norm{u_k-v_0}_{\BD{T}_p}\geq   \norm{U-v_0}_{\BD{T}_p}-\norm{u_k-U}_{\BD{T}_p}\geq \frac{3}{4}\bar{\rho}-\frac{1}{4}\bar{\rho}= \frac{1}{2}\bar{\rho}.
\end{equation}
By
(A), there exists an $i\in (\BD{m}_4+2,\BD{m}_4+l-2)$ such that
\begin{equation*}
  \norm{U-v_0}_{X_i}\leq \sigma,
\end{equation*}
and thus for large $k$,
\begin{equation*}
  \norm{u_k-v_0}_{X_i}\leq 2\sigma.
\end{equation*}

Choose $q_k>p$ satisfying
\begin{equation*}
  \norm{u_k-v_0}_{X_{q_k}}\leq \sigma.
\end{equation*}
Define
\begin{equation}\label{eq:7.15}
  h_k=\left\{
        \begin{array}{ll}
          v_0, & \BD{i}_1\leq i-r-1 \quad \textrm{ or } \quad \BD{i}_1\geq q_k+r+1, \\
          u_k, & i-r\leq \BD{i}_1\leq q_k+r.
        \end{array}
      \right.
\end{equation}
Thus $h_k\in \Gamma_1(v_0)\cap \hat{\Gamma}_1(v_0,w_0)$, and by \eqref{eq:7.15} and \eqref{eq:7.8}, for $k$ large enough
\begin{equation}\label{eq:7.16}
\begin{split}
  \beta(\frac{\bar{\rho}}{2})&\leq J_1(h_k)=J_{1;i-2r-1,q_k+2r+1}(h_k)\\
  &=J_{1;i-2r-1,i-1}(h_k)+J_{1;i,q_k}(u_k)+J_{1;q_k+1,q_k+2r+1}(h_k)\\
&\leq J_{1;i,q_k}(u_k)+\kappa(2\sigma)(4r+2),
\end{split}
\end{equation}
where $\beta$ is given by Proposition \ref{prop:6.13}.
By \eqref{eq:7.16}, we have
\begin{equation}\label{eq:7.17}
  J_1(u_k)\geq J_{1;-\infty,i-1}(u_k)+\beta(\frac{\bar{\rho}}{2})-\kappa(2\sigma)(4r+2)+J_{1;q_k+1,\infty}(u_k).
\end{equation}
Setting
\begin{equation*}%
  g_k=\left\{
        \begin{array}{ll}
          u_k, & \BD{i}_1\leq i-1 \textrm{ or } \BD{i}_1\geq q_k+2, \\
          v_0, & i\leq \BD{i}_1\leq q_k+1,
        \end{array}
      \right.
\end{equation*}
gives
\begin{equation*}
  |J_{1;-\infty,i-1}(u_k)-J_{1;-\infty,i-1}(g_k)|+|J_{1;i,q_k}(g_k)|+|J_{1;q_k+1,\infty}(u_k)-J_{1;q_k+1,\infty}(g_k)|\leq \kappa_0(\sigma)
\end{equation*}
with $\kappa_0(\theta)\to 0$ as $\theta\to 0$.
Therefore
\begin{equation}\label{eq:7.20}
  J_{1;-\infty,i-1}(u_k)+J_{1;q_{k}+1,\infty}(u_k)\geq J_1(g_k)-2\kappa_0(\sigma).
\end{equation}
Thus \eqref{eq:7.17} and \eqref{eq:7.20} yield
\begin{equation*}
  J_1(u_k)\geq J_1(g_k)+\beta(\frac{\bar{\rho}}{2})-\kappa(2\sigma)(4r+2)-2\kappa_0(\sigma).
\end{equation*}
Let $\sigma$ be small enough such that
\begin{equation*}
  \kappa(2\sigma)(4r+2)< \frac{1}{3} \beta(\frac{\bar{\rho}}{2}), \quad 2\kappa_0(\sigma)<\frac{1}{3} \beta(\frac{\bar{\rho}}{2}).
\end{equation*}
Thus \[J_1(u_k)\geq J_1(g_k)+\frac{1}{3}\beta(\frac{\bar{\rho}}{2})\geq b_{\BD{m},l}+\frac{1}{3}\beta(\frac{\bar{\rho}}{2}),\]
where the last inequality follows from $(g_k)\subset Y_{\BD{m},l}$.
But this is absurd since
$J_1(u_k)\to b_{\BD{m},l}$ as $k\to \infty$.\\

\emph{Proof of (C).}
$J_1(U)\geq b_{\BD{m},l}$ follows $U\in Y_{\BD{m},l}$.
A variant proof of \eqref{eq:claim789541} shows $J_1(U)\leq b_{\BD{m},l}$.
\\

\emph{Proof of (D).}
By the proof of \eqref{eq:claim11212}, to prove (D), it suffices to show that
there are strict inequalities in \eqref{eq:6.5} for $U$
provided that $\BD{m}_2-\BD{m}_1$, $\BD{m}_4-\BD{m}_3$ are large enough.
Suppose, by contradiction, there is some $i$, such that 
\begin{equation*}
  \norm{U-v_0}_{\BD{T}_i}=\rho_1\quad\quad\textrm{or}\quad\quad \norm{U-w_0}_{\BD{T}_i}=\rho_2.
\end{equation*}
By (A), there is a
$q\in [\BD{m}_3-l+2,\BD{m}_3-3]$ such that
\begin{equation}\label{eq:7.25}
  \norm{U-w_0}_{X_q}\leq \sigma.
\end{equation}
Set
\begin{equation*}
  \Phi=\left\{
         \begin{array}{ll}
           U, & \BD{i}_1\leq q, \\
           w_0, & q+1\leq \BD{i}_1,
         \end{array}
       \right.
\end{equation*}
and
\begin{equation*}
  \Psi=\left\{
         \begin{array}{ll}
           w_0, & \BD{i}_1\leq q, \\
           U, & q+1\leq \BD{i}_1,
         \end{array}
       \right.
\end{equation*}
Since $\tau_{-q}^{1}\Phi\in \Lambda_1(v_0,w_0)$ and $\Psi\in\Gamma_1(w_0,v_0)$, we have
\begin{equation}\label{eq:7.30}
  J_1(\Phi)=J_1(\tau_{-q}^{1}\Phi)\geq d_1(v_0,w_0) \quad\textrm{and}\quad J_1(\Psi)\geq c_1(w_0,v_0).
\end{equation}
By \eqref{eq:7.25}-\eqref{eq:7.30},
\begin{equation}\label{eq:7.33}
\begin{split}
  J_1(U)&\geq J_1(\Phi)+J_1(\Psi)-\kappa_0(\sigma)\\
  &\geq d_1(v_0,w_0)+c_1(w_0,v_0)-\kappa_0(\sigma).
  \end{split}
\end{equation}

Concatenating suitable minimal and Birkhoff solutions of \eqref{eq:PDE}, we can construct a configuration in $Y_{\BD{m},l}$, which gives an upper bound for $J_1(U)$.
To this end, choose $V_1\in \MM(v_0,w_0)$ and $W_1\in \MM(w_0,v_0)$ such that $\norm{V_1-w_0}_{X_q}\leq \sigma$, $\norm{W_1-w_0}_{X_q}\leq \sigma$ and \[J_{1;-\infty,q}(V_1)+J_{1;q+1,\infty}(W_1)\leq c_1(v_0,w_0)+c_1(w_0,v_0)+\epsilon.\] Here
$\epsilon=\epsilon(\BD{m}_2-\BD{m}_1,\BD{m}_4-\BD{m}_3)>0$ satisfies $\epsilon\to 0$ as $\BD{m}_2-\BD{m}_1,\BD{m}_4-\BD{m}_3\to \infty$.
Define
\begin{equation*}
  \hat{U}=\left\{
            \begin{array}{ll}
              V_1, & \BD{i}_1\leq q, \\
              W_1, & q+1\leq \BD{i}_1.
            \end{array}
          \right.
\end{equation*}
Then we obtain
\begin{equation}\label{eq:7.35}
  J_1(U)\leq J_1(\hat{U})\leq c_1(v_0,w_0)+c_1(w_0,v_0)+\epsilon+\kappa_0(\sigma),
\end{equation}
Hence \eqref{eq:7.33}-\eqref{eq:7.35} and Proposition \ref{prop:6.74} imply
\begin{equation*}
  0<d_1(v_0,w_0)-c_1(v_0,w_0)\leq \epsilon +2\kappa_0(\sigma).
\end{equation*}
But $\epsilon$ and $\kappa_0(\sigma)$ can be taken arbitrary small, a contradiction.
Similarly, one can prove that there will not hold equalities in \eqref{eq:6.5} (c), (d), provided that $\BD{m}_2-\BD{m}_1,\BD{m}_4-\BD{m}_3$ are large enough.
This proves (D) and thus Theorem \ref{thm:6.8}. 
\qed

\begin{rmk}\label{rem:7.39}
As stated in Section \ref{sec:intro}, we obtain infinitely many geometrically distinct solutions of \eqref{eq:PDE}.
Indeed, noticing the dependence of the solution $U$ given by Theorem \ref{thm:6.8} on $l,\BD{m}$ and $\rho_i$ ($1\leq i \leq 4$),
there will be infinitely many solutions if we fix $\rho_i$ and enlarge $l, \BD{m}_2-\BD{m}_1, \BD{m}_4-\BD{m}_3$.
\end{rmk}

\begin{rmk}
Since the solution $U$ obtained in Theorem \ref{thm:6.8} is not Birkhoff, by Proposition \ref{prop:6.93}, $U$ is not minimal.
\end{rmk}

Since lacking of the properties of Birkhoff and minimum, the structure of the set of solutions given by Theorem \ref{thm:6.8}
is difficult to analyze.
Rabinowitz and Stredulinsky
proved for fixed $\rho_i$ ($1\leq i\leq  4$), there is an ordered pair of solutions with different parameters $l,\BD{m}$.
In our setting, we have a similar result, but we shall not state and prove it here.
The interested reader is referred to \cite[Corollary 7.40]{RS} for this result.

\section{Generalizations}\label{chap:8}

We give some generalizations in this section.
We only state the necessary changes of the variational problems and the corresponding theorems but without proofs.
The proofs of these results follow as that of Theorem \ref{thm:6.8} in Section \ref{chap:7} with slight modifications.
Throughout this section, we assume $\rho_i$ ($1\leq i\leq 4$) are defined as in \eqref{eq:6.2}.

\subsection{Homoclinic solutions to $w_0$.}\label{sec4:4.1}
\

The first generalization is solutions homoclinic to $w_0$ as $|\BD{i}_1|\to\infty$ and periodic in $\BD{i}_2, \cdots, \BD{i}_n$.
Comparing to the variational problem \eqref{eq:6.4}-\eqref{eq:6.7}, 
we have to modify \eqref{eq:6.5} and \eqref{eq:6.6} as follows.
Suppose $\BD{m}$ satisfies \eqref{eq:6.3}.
If we replace
\eqref{eq:6.5} and \eqref{eq:6.6} by
\begin{equation*}
  \left\{
    \begin{array}{ll}
(a) \quad\rho_{+}(\tau_{-i}^{1}u)&\leq \rho_3,\quad \BD{m}_1-l\leq i\leq \BD{m}_1-1, \\
(b) \quad\rho_{-}(\tau_{-i}^{1}u)&\leq \rho_4, \quad \BD{m}_2\leq i\leq \BD{m}_2+l-1, \\
(c) \quad\rho_{-}(\tau_{-i}^{1}u)&\leq \rho_1,\quad \BD{m}_3-l\leq i\leq \BD{m}_3-1, \\
(d) \quad\rho_{+}(\tau_{-i}^{1}u)&\leq \rho_2, \quad \BD{m}_4\leq i\leq \BD{m}_4+l-1,
    \end{array}
  \right.
\end{equation*}
and
\begin{equation*}
  \norm{u-w_0}_{\BD{T}_i}\to 0,\quad |i|\to \infty,
\end{equation*}
respectively, then we obtain a
theorem for this case that is same to Theorem \ref{thm:6.8} without any modification.

\subsection{Homoclinic solutions: $2k$ transition solutions.}\label{sec4:4.2}
\

The second is $2k$ ($k>1$) transition solutions.
Suppose $\BD{m}=(\BD{m}_1,\BD{m}_2,\cdots,\BD{m}_{4k})\in \Z^{4k}$ satisfies
\begin{equation}\label{eq:mm11}
  \BD{m}_i<\BD{m}_{i+1} \quad \textrm{and} \quad \BD{m}_{j}+2l<\BD{m}_{j+1} \quad \textrm{for $j$ even}.
\end{equation}
Comparing to the variational problem \eqref{eq:6.4}-\eqref{eq:6.7}, what is needed to modify is \eqref{eq:6.5}.
\eqref{eq:6.5} should be replaced by
\begin{equation*}
  \left\{
    \begin{array}{ll}
(a) \quad\rho_{-}(\tau_{-i}^{1}u)&\leq \rho_1,\quad \BD{m}_{1+4j}-l\leq i\leq \BD{m}_{1+4j}-1,\quad j=0,1,\cdots,k-1, \\
(b) \quad\rho_{+}(\tau_{-i}^{1}u)&\leq \rho_2, \quad \BD{m}_{2+4j}\leq i\leq \BD{m}_{2+4j}+l-1, \quad j=0,1,\cdots,k-1,\\
(c) \quad\rho_{+}(\tau_{-i}^{1}u)&\leq \rho_3,\quad \BD{m}_{3+4j}-l\leq i\leq \BD{m}_{3+4j}-1, \quad j=0,1,\cdots,k-1,\\
(d) \quad\rho_{-}(\tau_{-i}^{1}u)&\leq \rho_4, \quad \BD{m}_{4+4j}\leq i\leq \BD{m}_{4+4j}+l-1, \quad j=0,1,\cdots,k-1.
    \end{array}
  \right.
\end{equation*}
We obtain:
\begin{thm}
Assume that $s\in C^2(\R^{B^{r}_{\BD{0}}},\R)$ satisfies (S\ref{eq:S1})-(S\ref{eq:S3}), $k\geq 2$, and \eqref{eq:*0} \eqref{eq:*1} hold.
If $l\gg 0$, there is a $U\in Y_{\BD{m},l}$ such that $J_1(U)=b_{\BD{m},l}=\inf_{Y_{\BD{m},l}}J_1$.
Moreover, $U$ is a solution of \eqref{eq:PDE} provided that $\BD{m}_2-\BD{m}_1,\cdots,\BD{m}_{4k}-\BD{m}_{4k-1}$ are large enough.
\end{thm}
\begin{rmk}
As in Section \ref{sec4:4.1}, we also have homoclinic solutions asymptotic to $w_0$ as $|\BD{i}_1|\to \infty$.
We omit the statement of
the corresponding result.
\end{rmk}

\subsection{Heteroclinic solutons: $2k+1$ transition solutions.}\label{sec4:4.3}
\

The third is heteroclinic solution in $\hat{\Gamma}_1(v_0,w_0)$ that asymptotic to $v_0$ and $w_0$ as $\BD{i}_1\to -\infty$ and $\BD{i}_1\to \infty$, respectively.
This solution is $2k+1$ ($k\geq 1$) transition solution.
Suppose $\BD{m}=(\BD{m}_1,\BD{m}_2,\cdots,\BD{m}_{4k+2})\in\Z^{4k+2}$ satisfies \eqref{eq:mm11}.
Comparing to the variational problem \eqref{eq:6.4}-\eqref{eq:6.7},
we have to modify \eqref{eq:6.5} and \eqref{eq:6.6} as follows:
\begin{equation}\label{eq:6.5333333}
  \left\{
    \begin{array}{ll}
(a) \quad\rho_{-}(\tau_{-i}^{1}u)&\leq \rho_1,\quad \BD{m}_{1+4j}-l\leq i\leq \BD{m}_{1+4j}-1,\quad j=0,1,\cdots,k, \\
(b) \quad\rho_{+}(\tau_{-i}^{1}u)&\leq \rho_2, \quad \BD{m}_{2+4j}\leq i\leq \BD{m}_{2+4j}+l-1, \quad j=0,1,\cdots,k,\\
(c) \quad\rho_{+}(\tau_{-i}^{1}u)&\leq \rho_3,\quad \BD{m}_{3+4j}-l\leq i\leq \BD{m}_{3+4j}-1, \quad j=0,1,\cdots,k-1,\\
(d) \quad\rho_{-}(\tau_{-i}^{1}u)&\leq \rho_4, \quad \BD{m}_{4+4j}\leq i\leq \BD{m}_{4+4j}+l-1, \quad j=0,1,\cdots,k-1.
    \end{array}
  \right.
\end{equation}
and
\begin{equation}\label{eq:6.633335}
\begin{split}
  \norm{u-v_0}_{\BD{T}_i}&\to 0,\quad i\to -\infty, \\
  \norm{u-w_0}_{\BD{T}_i}&\to 0,\quad i\to \infty.
  \end{split}
\end{equation}
We obtain:
\begin{thm}
Let $s\in C^2(\R^{B^{r}_{\BD{0}}},\R)$ satisfies (S\ref{eq:S1})-(S\ref{eq:S3}), $k\geq 1$, and \eqref{eq:*0} \eqref{eq:*1} hold.
If $l\gg 0$, there is a $U\in Y_{\BD{m},l}$ such that $J_1(U)=b_{\BD{m},l}=\inf_{Y_{\BD{m},l}}J_1$.
Moreover, $U$ is a solution of \eqref{eq:PDE} provided that $\BD{m}_2-\BD{m}_1,\cdots,\BD{m}_{4k+2}-\BD{m}_{4k+1}$ are large enough.
\end{thm}

\begin{rmk}
Interchanging $v_0,w_0$ in \eqref{eq:6.633335} and modifying \eqref{eq:6.5333333} suitably
will give the heteroclinic solution from $w_0$ to $v_0$ in $\BD{i}_1$.
\end{rmk}

\subsection{Multitransition solutions in higher dimension.}\label{sec4:4.4}
\

In \cite[Section 4]{LC}, we construct solutions of \eqref{eq:PDE} heteroclinic in $\BD{i}_1$, $\BD{i}_2$ and periodic in $\BD{i}_3,\cdots \BD{i}_n$ under gap conditions \eqref{eq:*0}, \eqref{eq:*1}.
We denote by $\MM_2(v_1, w_1)$ the solutions heteroclinic in $\BD{i}_2$ from $v_1$ to $w_1$, and by $\MM_2(w_1,v_1)$ the ones heteroclinic in $\BD{i}_2$ from $w_1$ to $v_1$.
In particular, $\MM_2(v_1, w_1)$ and $\MM_2(w_1,v_1)$ are ordered sets.
Please see \cite[Section 4]{LC} for more details.

Using the constrained variational method of the present paper, we can obtain multitransition solutions lying between $v_1$ and $w_1$ provided
\begin{equation}\label{eq:*2}
\begin{split}
    \textrm{there are adjacent } &v_2, w_2 \in \MM_1(v_1,w_1) \textrm{ with } v_2<w_2, \\
    \textrm{ and there are adjacent } &\tilde{v}_2, \tilde{w}_2 \in \MM_1(w_1,v_1) \textrm{ with } \tilde{v}_2<\tilde{w}_2 .
    \end{split}\tag{$*_2$}
\end{equation}
For example, $2$ transition solutions lying between the gap of $v_1$ and $w_1$ can be obtained as follows.
Define
\[
\hat{\Gamma}_{2}(v_1,w_1):=\{u\in\R^{\Z\times \Z\times (\Z/\{1\})^{n-2}}\,|\, v_1\leq u\leq w_1\}.
\]
Set $\bar{\rho}:=\norm{w_1-v_1}_{E_0}$, $\rho_{-}(u):=\norm{u-v_1}_{E_0}$, and $\rho_{+}(u):=\norm{u-w_1}_{E_0}$.
Take $\rho_i\in (0, \bar{\rho})$, $1\leq i\leq 4$, satisfying
\begin{equation*}
\begin{split}
  \rho_1\not\in \rho_{-}(\MM_2(v_1,w_1)), & \quad\quad\rho_2\not\in\rho_{+}(\MM_2(v_1,w_1)), \\
   \rho_3\not\in \rho_{+}(\MM_2(w_1,v_1)), & \quad\quad\rho_4\not\in\rho_{-}(\MM_2(w_1,v_1)).
\end{split}
\end{equation*}
Let
\begin{equation*}
Y_{\BD{m},l}:= Y_{\BD{m},l}(v_0,w_0):= \{u\in \hat{\Gamma}_2(v_1,w_1)\,|\, u \textrm{ satisfies } \eqref{eq:4-6.5}-\eqref{eq:4-6.6}\},
\end{equation*}
where
\begin{equation}\label{eq:4-6.5}
  \left\{
    \begin{array}{ll}
(a) \quad\rho_{-}(\tau_{-i}^{2}u)&\leq \rho_1,\quad \BD{m}_1-l\leq i\leq \BD{m}_1-1, \\
(b) \quad\rho_{+}(\tau_{-i}^{2}u)&\leq \rho_2, \quad \BD{m}_2\leq i\leq \BD{m}_2+l-1, \\
(c) \quad\rho_{+}(\tau_{-i}^{2}u)&\leq \rho_3,\quad \BD{m}_3-l\leq i\leq \BD{m}_3-1, \\
(d) \quad\rho_{-}(\tau_{-i}^{2}u)&\leq \rho_4, \quad \BD{m}_4\leq i\leq \BD{m}_4+l-1,
    \end{array}
  \right.
\end{equation}
and
\begin{equation}\label{eq:4-6.6}
  \norm{u-v_0}_{E_i}\to 0,\quad |i|\to \infty. 
\end{equation}
Set
\begin{equation}\label{eq:4-6.7}
  b_{\BD{m},l}:= b_{\BD{m},l}(v_1,w_1):=\inf_{u\in Y_{\BD{m},l}}J_{2}(u).
\end{equation}
We have:
\begin{thm}\label{thm:4-6.8}
Suppose $s\in C^2(\R^{B^{r}_{\BD{0}}},\R)$ satisfies (S\ref{eq:S1})-(S\ref{eq:S3}).
Assume that \eqref{eq:*0}, \eqref{eq:*1} and \eqref{eq:*2} hold.
Then for each sufficiently large $l\in\N$, there is a $U=U_{m,l}\in Y_{\BD{m},l}$ such that $J_2(U)=b_{\BD{m},l}$.
Moreover, $U$ is a solution of \eqref{eq:PDE} provided that $\BD{m}_2-\BD{m}_1$ and $\BD{m}_4-\BD{m}_3$ are large enough.
\end{thm}

\begin{rmk}
The results of Sections \ref{sec4:4.1}-\ref{sec4:4.3} can be easily generalized as above.
\end{rmk}
\begin{rmk}
More multitransition solutions in higher dimension can be obtained under more gap conditions, cf. \cite[Section 5.1]{LC}.
\end{rmk}

\subsection{Multitransition solutions in other coordinate systems.}\label{sec4:4.5}
\

In \cite[Section 5.2]{LC}, we know that changing coordinate system will not produce more periodic solutions, but it will give more heteroclinic solutions.
So multitransition solutions as in Theorem \ref{thm:6.8} and Sections \ref{sec4:4.1}-\ref{sec4:4.4}
can be obtained in the new coordinate system.
The interested reader is referred to \cite[Section 5.2]{LC} (see also \cite[Section 5.2]{RS}).

\subsection{Multitransition solutions constructed by heteroclinic solutions of rotation vector $\alpha\in\Q\setminus \{\BD{0}\}$.}\label{sec4:4.6}
\

If we replace the
rotation vector $\alpha=\BD{0}$ by $\alpha\in\Q\setminus \{\BD{0}\}$, we can obtain more multitransition solutions.
Indeed, in \cite[Section 5.3]{LC}, we construct periodic solutions and basic heteroclinic solutions corresponding to rotation vector $\alpha\in\Q\setminus \{\BD{0}\}$
by transferring this problem into a problem of the form as in \cite[Section 5.2]{LC}.
Thus using the same idea,
by Section \ref{sec4:4.5} we have multitransition solutions lying between periodic solutions corresponding to $\alpha\in\Q\setminus \{\BD{0}\}$,
and multitransition solutions lying between basic heteroclinic solutions corresponding to $\alpha\in\Q\setminus \{\BD{0}\}$.

\subsection*{Acknowledgments}
Wen-Long Li is supported by the Fundamental Research Funds for the Central Universities (no. 34000-31610274).
Xiaojun Cui is supported by the National Natural Science Foundation
of China (Grants 11571166, 11631006, 11790272), the Project Funded by
the Priority Academic Program Development of Jiangsu Higher Education
Institutions (PAPD) and the Fundamental Research Funds for the
Central Universities.


\end{CJK*}
\end{document}